\documentclass[11pt]{article}

\usepackage{amsmath}
\usepackage{amsthm}
\usepackage{amsfonts}
\usepackage{bbm}
\usepackage{graphicx}
\usepackage{color}
\usepackage[noblocks,affil-it]{authblk}

\newcommand{\mmp}{\mathbb{P}}

\newcommand{\dod}{\overset{d}{\to}}
\newcommand{\tp}{\overset{\mathbb{P}}{\to}}

\newcommand{\me}{\mathbb{E}}
\newcommand{\mr}{\mathbb{R}}
\newcommand{\mn}{\mathbb{N}}

\newtheorem{thm}{Theorem}[section]
\newtheorem{lemma}[thm]{Lemma}

\newtheorem{cor}[thm]{Corollary}

\newtheorem{assertion}[thm]{Proposition}
\theoremstyle{definition}

\theoremstyle{remark}
\newtheorem{rem}[thm]{Remark}

\DeclareMathOperator{\1}{\mathbbm{1}}
\begin{document}
\title{Fractionally integrated inverse stable subordinators}

\date{\today}
\author[1]{Alexander Iksanov\thanks{E-mail: iksan@univ.kiev.ua}}
\affil[1]{Faculty of Cybernetics, Taras Shevchenko National University of Kyiv, 01601 Kyiv, Ukraine}
\author[2]{Zakhar Kabluchko\thanks{E-mail: zakhar.kabluchko@uni-muenster.de}}
\affil[2]{Institut f\"{u}r Mathematische Statistik, Westf\"{a}lische Wilhelms-Universit\"{a}t M\"{u}nster, 48149 M\"{u}nster, Germany}
\author[1,2]{Alexander Marynych\thanks{E-mail: marynych@unicyb.kiev.ua}}
\author[3]{Georgiy Shevchenko\thanks{E-mail: zhora@univ.kiev.ua}}
\affil[3]{Faculty of Mechanics and Mathematics, Taras Shevchenko National University
of Kyiv, 01601 Kyiv, Ukraine}

\maketitle

\begin{abstract}
A fractionally integrated inverse stable subordinator (FIISS) is
the convolution of a power function and an inverse stable
subordinator. We show that the FIISS is a scaling limit in the
Skorokhod space of a renewal shot noise process with heavy-tailed,
infinite mean `inter-shot' distribution and regularly varying
response function. We prove local H\"{o}lder continuity of FIISS
and a law of iterated logarithm for both small and large times.
\end{abstract}

%   Central limit and other weak theorems; renewal theory
\hphantom{\emph{2010 Mathematics Subject Classification: }}Secondary          60G18               %   Self-similar processes


\section{Introduction}

\subsection{A brief survey of inverse stable subordinators}\label{W}

For $\alpha\in (0,1)$, let $(D_{\alpha}(t))_{t \geq 0}$ be an
$\alpha$-stable subordinator, i.e., an increasing L\'evy process,
with\footnote{We write $\Gamma(1-\alpha)t^\alpha$ rather than just
$t^\alpha$ to conform with the notation exploited in our previous
works.} $-\log \me e^{-t D_{\alpha}(1)} = \Gamma(1-\alpha)
t^\alpha$ for $t \geq 0$, where $\Gamma(\cdot)$ is Euler's gamma
function. Its generalized inverse
$W_\alpha:=(W_{\alpha}(u))_{u\in\mr}$ defined by
\begin{equation*}
W_{\alpha}(u)    ~:=~    \inf\{t \geq 0: D_{\alpha}(t)>u\}, \quad
u \geq 0
\end{equation*}
and $W_\alpha(u):=0$ for $u<0$, is called an {\it inverse
$\alpha$-stable subordinator}. Obviously, $W_\alpha$ has a.s.\
continuous and nondecreasing sample paths. Further, it is clear
that $W_\alpha$ is self-similar with index $\alpha$, i.e., the
finite-dimensional distributions of $(W_\alpha(cu))_{u\geq 0}$ for
fixed $c>0$ are the same as those of $(c^\alpha
W_\alpha(u))_{u\geq 0}$.

More specific properties of $W_\alpha$ include (local) H\"{o}lder
continuity with arbitrary exponent $\gamma<\alpha$ which is a
consequence of
\begin{equation}\label{sam}
M:=\sup_{0\leq v<u\leq 1/2}\,\frac{W_\alpha(u)-W_\alpha(v)}{
(u-v)^\alpha|\log (u-v)|^{1-\alpha}}<\infty\quad\text{a.s.}
\end{equation}
(Lemma 3.4 in \cite{Owada+Samorodnitsky:2014+}), a modulus
of continuity result
\begin{equation*}\label{modcont}
\underset{\delta\to 0+}{\lim}\,\underset{\substack{0\leq t\leq
1\\0<h<\delta}}{\sup}\,\frac{W_\alpha(t+h)-W_\alpha(t)}{
h^\alpha |\log h|^{1-\alpha}}=\frac{1}{
\Gamma(1-\alpha)\alpha^{2\alpha-1}(1-\alpha)^{1-\alpha}}
\quad\text{a.s.}
\end{equation*}
(formula (6) in \cite{Hawkes:1971}), and the law of iterated
logarithm
\begin{equation}\label{bert}
\lim\sup\,\frac{W_\alpha(u)}{u^\alpha(\log|\log
u|)^{1-\alpha}}=\frac{1}{
\Gamma(1-\alpha)\alpha^\alpha(1-\alpha)^{1-\alpha}}\quad\text{a.s.}
\end{equation}
both as $u\to 0+$ and $u\to+\infty$ which can be extracted from
Theorem 4.1 in \cite{Bertoin:1999}. For later needs, we note that
the random variable $M$ defined in \eqref{sam} satisfies
\begin{equation}\label{sam3}
\me e^{sM}<\infty
\end{equation}
for all $s>0$ (Lemma 3.4 in \cite{Owada+Samorodnitsky:2014+}).

Denote by $D[0,\infty)$ and $D(0,\infty)$ the Skorokhod spaces of
right-continuous real-valued functions which are defined on
$[0,\infty)$ and $(0,\infty)$, respectively, and have finite
limits from the left at each positive point. Elements of these
spaces are sometimes called {\it c\`{a}dl\`{a}g} functions.
Throughout the paper, weak convergence on $D[0,\infty)$ or
$D(0,\infty)$ endowed with the well-known $J_1$-topology is
denoted by $\Rightarrow$. See \cite{Billingsley:1999, Whitt:2002}
for a comprehensive account on the $J_1$-topology.

Let $\xi_1$, $\xi_2,\ldots$ be a sequence of independent copies of
a positive random variable $\xi$. Denote by $(S_n)_{n\in\mn_0}$,
where $\mn_0:=\mn\cup\{0\}$, the zero-delayed standard random walk
with jumps $\xi_k$, i.e., $S_0:=0$ and $S_n:=\xi_1+\ldots+\xi_n$
for $n\in\mn$. The corresponding first-passage time process is
defined by
$$\nu(t):=\inf\{k\in\mn_0: S_k>t\},\quad t\in\mr.$$ Note that
$\nu(t)=0$ for $t<0$.

Assume that
\begin{equation}\label{1}
\mmp\{\xi>t\}\quad\sim\quad t^{-\alpha}\ell(t),\quad t\to\infty
\end{equation}
for some $\alpha\in (0,1)$ and some $\ell$ slowly varying at $\infty$. Then, according to Corollary 3.4 in \cite{Meerschaert+Scheffler:2004},
\begin{equation}\label{flt_nu}
\mmp\{\xi>t\}\nu(ut)\quad\Rightarrow\quad W_\alpha(u),\quad t\to\infty
\end{equation}
on $D[0,\infty)$.

In the recent years inverse stable subordinators, also known as
Mittag-Leffler processes\footnote{The terminology stems from the
fact that, for any fixed $u>0$, the random variable $W_\alpha(u)$
has a Mittag-Leffler distribution with parameter $\alpha$, see
Section \ref{Lamp} below.}, have become a popular object of
research, both from the theoretical and applied viewpoints.
Relation \eqref{flt_nu} which tells us that the processes
$W_\alpha$ are scaling limits of the first-passage time processes
with heavy-tailed waiting times underlies the ubiquity of inverse stable subordinators in a heavy-tailed world. For instance, inverse stable subordinators are often used as a time-change of the subordinated processes intended to model heavy-tailed phenomena. %They usually arise in the probabilistic models involving heavy-tailed phenomena in the role of a time-change in the underlying process governing the evolution of a model.
The most prominent example of this kind is a scaling limit for
continuous-time random walks  with heavy-tailed waiting times
\cite{Meerschaert+Scheffler:2004,Meerschaert+Straka:2013}. In the
simplest situation, the scaling limit takes the form
$S(W_{\alpha}(\cdot))$, where $S(\cdot)$ is a $\gamma$-stable
process with $0<\gamma\leq 2$. The special case $\gamma=2$ appears
in many problems related to the {\it anomalous (or fractional)
diffusion} and has attracted considerable attention in both
physics \cite{Magdziarz+Weron:2011,Stanislavsky+Weron+Weron:2008}
and mathematics literature
\cite{Baeumer+Meerschaert+Nane:2009,Magdziarz+Schilling:2015,Nane:2009}.
More general subordinated processes $X(W_{\alpha}(\cdot))$, with
$X$ being a Markov process, can be used to construct solutions to
fractional partial differential equations
\cite{Meerschaert+Benson+Scheffler+Baeumer:2002,Meerschaert+Nane+Vellaisamy:2009}.
Also, inverse stable subordinators play an important role in the
analysis of (a) stationary infinitely divisible processes
generated by conservative flows
\cite{Jung+Owada+Samorodnitsky:2016+,Owada+Samorodnitsky:2014+}
and (b) asymptotics of convolutions of certain (explicitly given)
functions and rescaled continuous-time random walks
\cite{Scalas+Viles:2014}. In (a) and (b), the limit processes are
convolutions involving inverse stable subordinators and, as such,
are close relatives of processes $Y_{\alpha,\,\beta}$ to be
introduced below.

\subsection{Definition and known properties of fractionally integrated inverse stable
subordinators}\label{properties}

In this section we define the processes which are in focus in the
present paper and review some of their known properties.

For $\beta\in\mr$, set
\begin{equation*}
Y_{\alpha,\,\beta}(0):=0,\quad Y_{\alpha,\,\beta}(u) :=
\int_{[0,\,u]} (u-y)^\beta {\rm d}W_\alpha(y), \quad u>0.
\end{equation*}
Since the integrator $W_\alpha$ has nondecreasing paths, the
integral exists as a pathwise Lebesgue-Stieltjes integral.
Proposition \ref{finite_lemma} below shows that
$Y_{\alpha,\,\beta}(u)<\infty$ a.s.\ for each fixed $u>0$.
Following \cite{Iksanov:2013} and
\cite{Iksanov+Marynych+Meiners:2014}, we call
$Y_{\alpha,\,\beta}:=(Y_{\alpha,\,\beta}(u))_{u\geq 0}$ {\it
fractionally integrated inverse $\alpha$-stable subordinator}.

In \cite{Iksanov:2013}, it was shown that the processes
$Y_{\alpha,\,\beta}$ with $\beta\geq 0$ are scaling limits in the
Skorokhod space of renewal shot noise processes with eventually
nondecreasing regularly varying response functions and
heavy-tailed `inter-shot' distributions of infinite mean.
According to Theorem 2.9 in \cite{Iksanov+Marynych+Meiners:2014},
in the case when $\beta\in [-\alpha, 0]$ (and the response
functions are eventually nonincreasing) a similar statement holds
in the sense of weak convergence of  finite-dimensional
distributions. More exotic processes involving $Y_{\alpha,\,\beta}$ arise as scaling limits for random processes with immigration which are renewal shot noise processes with {\it random} response functions (see \cite{Iksanov+Marynych+Meiners:2016} for the precise definition). In Proposition 2.2 of \cite{Iksanov+Marynych+Meiners:2016}, the limit is a conditionally %(given $W_\alpha$)
Gaussian process with conditional variance $Y_{\alpha,\,\beta}$. %This is for aesthetic reasons. I believe it should be completely clear what is the condition.

We shall use the representations
\begin{equation}\label{repr1}
Y_{\alpha,\,\beta}(u)=\beta\int_0^u (u-y)^{\beta-1}W_\alpha(y){\rm
d}y,\quad u>0
\end{equation}
when $\beta > 0$ and
\begin{eqnarray}\label{repr2}
Y_{\alpha,\,\beta}(u)&=&u^\beta
W_\alpha(u)+|\beta|\int_0^u(W_\alpha(u)-W_\alpha(u-y))y^{\beta-1}{\rm
d}y\notag\\&=&|\beta|\int_0^\infty
(W_\alpha(u)-W_\alpha(u-y))y^{\beta-1}{\rm d}y,\quad u>0
\end{eqnarray}
when $-\alpha<\beta<0$. These show that $Y_{\alpha,\,\beta}$ is
nothing else but the {\it Riemann-Liouville fractional integral}
(up to a multiplicative constant) of $W_\alpha$ in the first case
and the {\it Marchaud fractional derivative} of $W_\alpha$ in the
second (see p.~33 and p.~111 in
\cite{Samko+Kilbas+Marichev:1993}).

Here are some known properties of $Y_{\alpha,\,\beta}$.

\begin{itemize}
\item[(I)]
%For each $u>0$ we have
$Y_{\alpha,\,\beta}(u)<\infty$ a.s.\ for each $u>0$ (the case
$\beta\geq 0$ is trivial; the case $\beta\in(-\alpha,0)$ is
covered by Lemma 2.14 in \cite{Iksanov+Marynych+Meiners:2013}; for
arbitrary $\beta$, see Proposition \ref{finite_lemma} below).
\item[(II)] $Y_{\alpha,\,\beta}$ is a.s.\ continuous whenever $\beta>-\alpha$ (see p.~1993 in \cite{Iksanov:2013}
for the case $\beta \geq 0$ and Proposition 2.18 in
\cite{Iksanov+Marynych+Meiners:2013} for the case $\beta\in
(-\alpha, 0)$). In the case when $\beta\leq -\alpha$ the
probability that $Y_{\alpha,\,\beta}$ is unbounded on a given
interval is strictly positive (see the proof of Proposition 2.7 in
\cite{Iksanov+Marynych+Meiners:2016} for the case $\beta=-\alpha$;
although an extension to the case $\beta<-\alpha$ is
straightforward, it is discussed in the proof of Proposition
\ref{finite_lemma} below for the sake of completeness).  %\textcolor{red}{but for the sake
%of completeness is given in the proof of Proposition
%\ref{finite_lemma} below}).
\item[(III)] The increments of $Y_{\alpha,\,\beta}$ are neither independent nor stationary (see p.~1994 in \cite{Iksanov:2013} and Proposition 2.16 in \cite{Iksanov+Marynych+Meiners:2013}).
\item[(IV)]
$Y_{\alpha,\,\beta}$ is self-similar with index $\alpha+\beta$ (even though this can be easily checked, we state this observation as Proposition \ref{simil} for ease of reference).
\end{itemize}

Three realizations of inverse $3/4$-stable subordinators together
with the corresponding fractionally integrated inverse
$3/4$-stable subordinators for different $\beta$ are shown on
Figure \ref{FIISS_fig1}.

\begin{figure}[!ht]
\begin{center}
\includegraphics[scale=0.5]{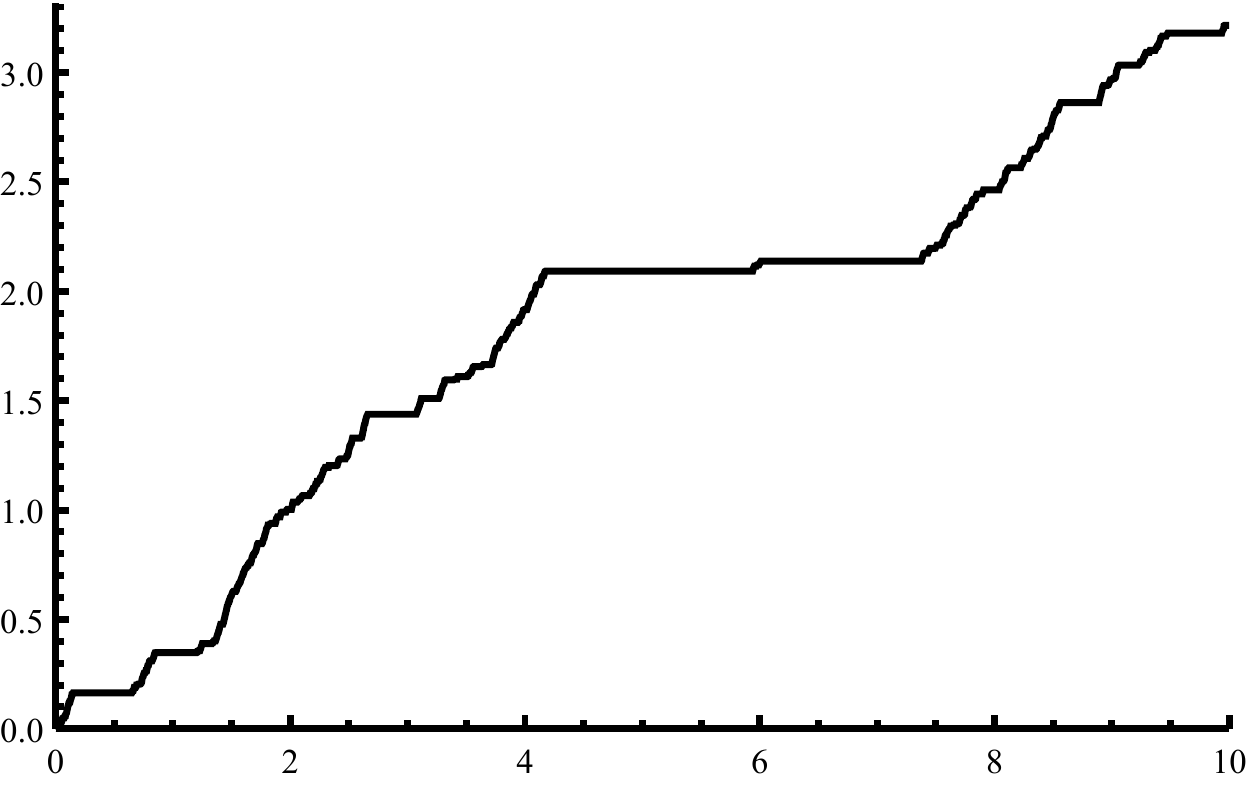}\hspace{1cm}\includegraphics[scale=0.5]{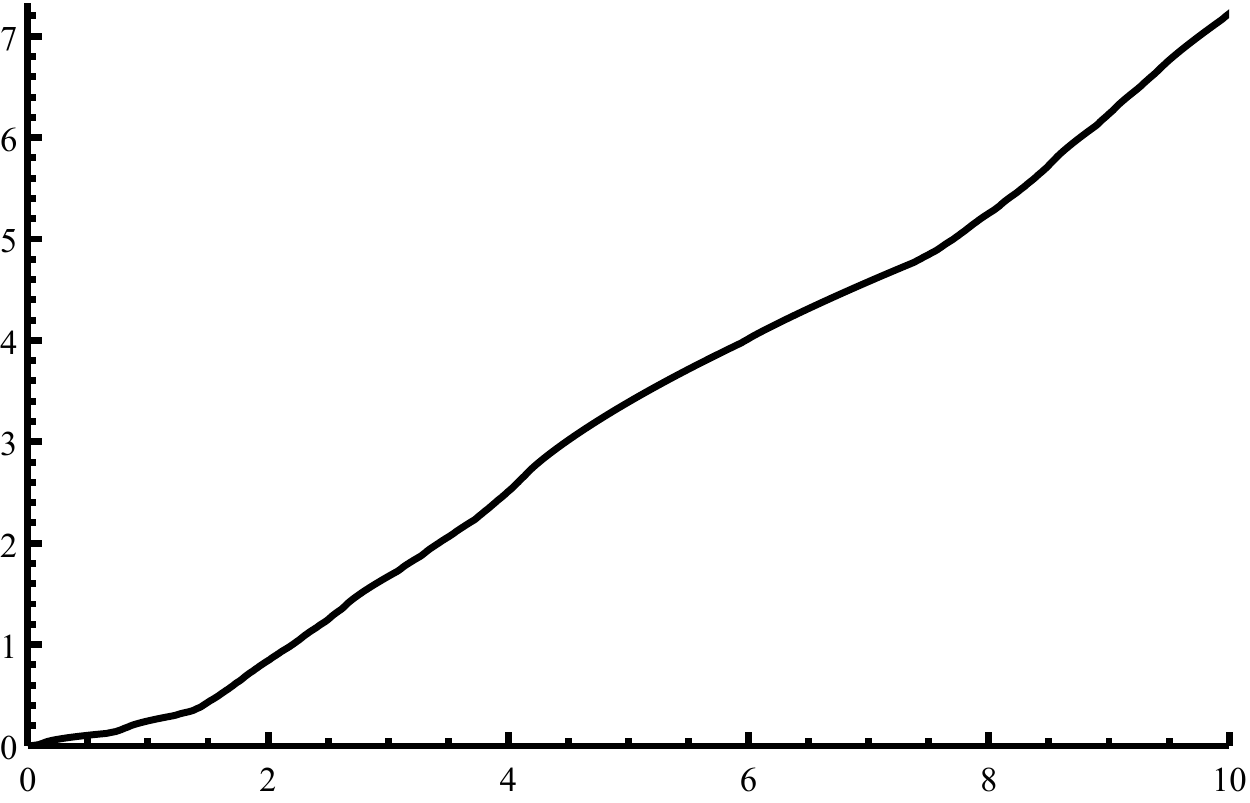}\\
$\alpha=0.75$ and $\beta=0.5$\\
\includegraphics[scale=0.5]{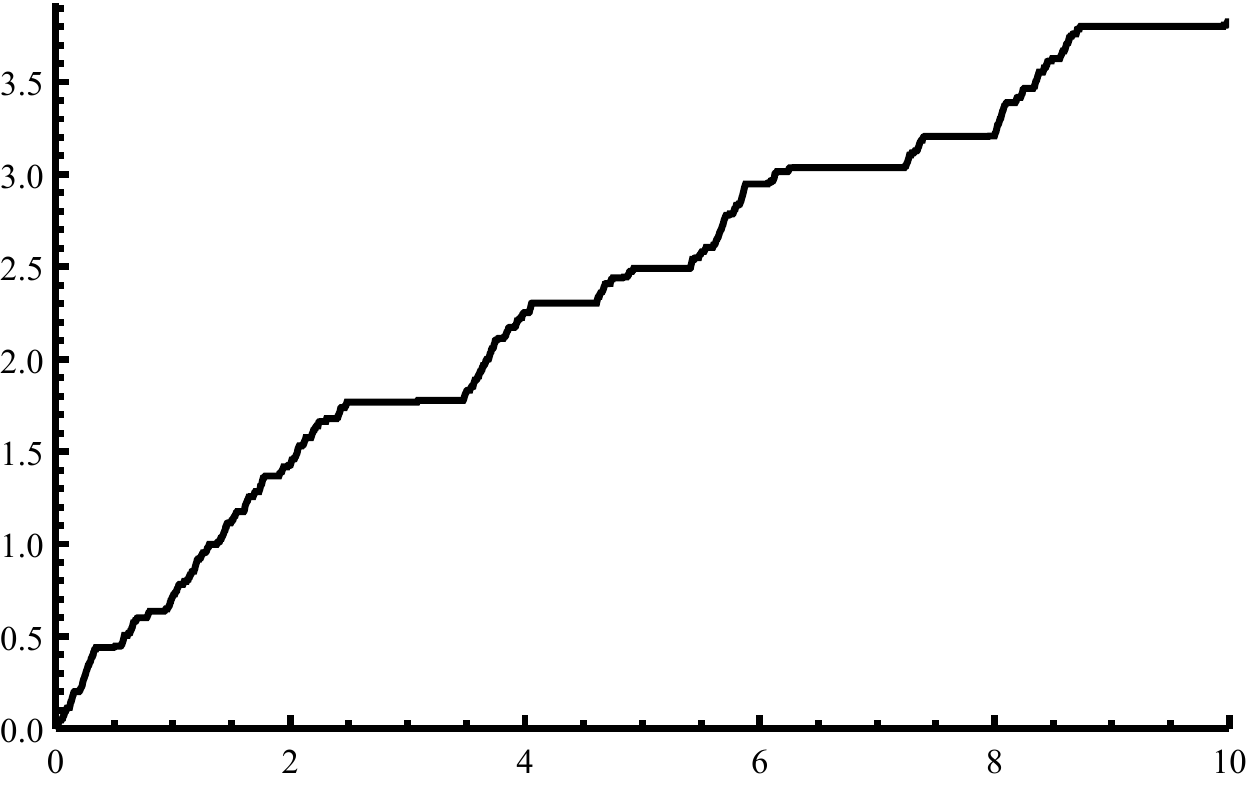}\hspace{1cm}\includegraphics[scale=0.5]{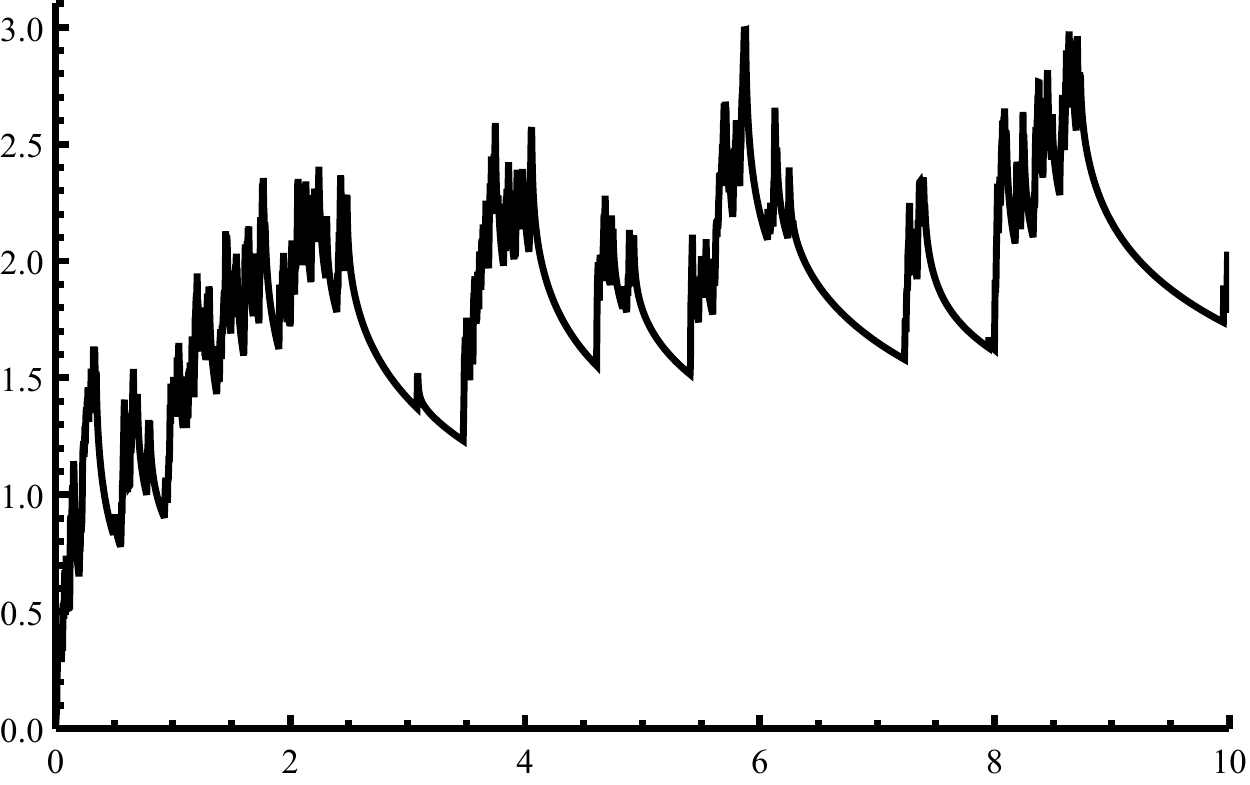}\\
$\alpha=0.75$ and $\beta=-0.5$\\
\includegraphics[scale=0.5]{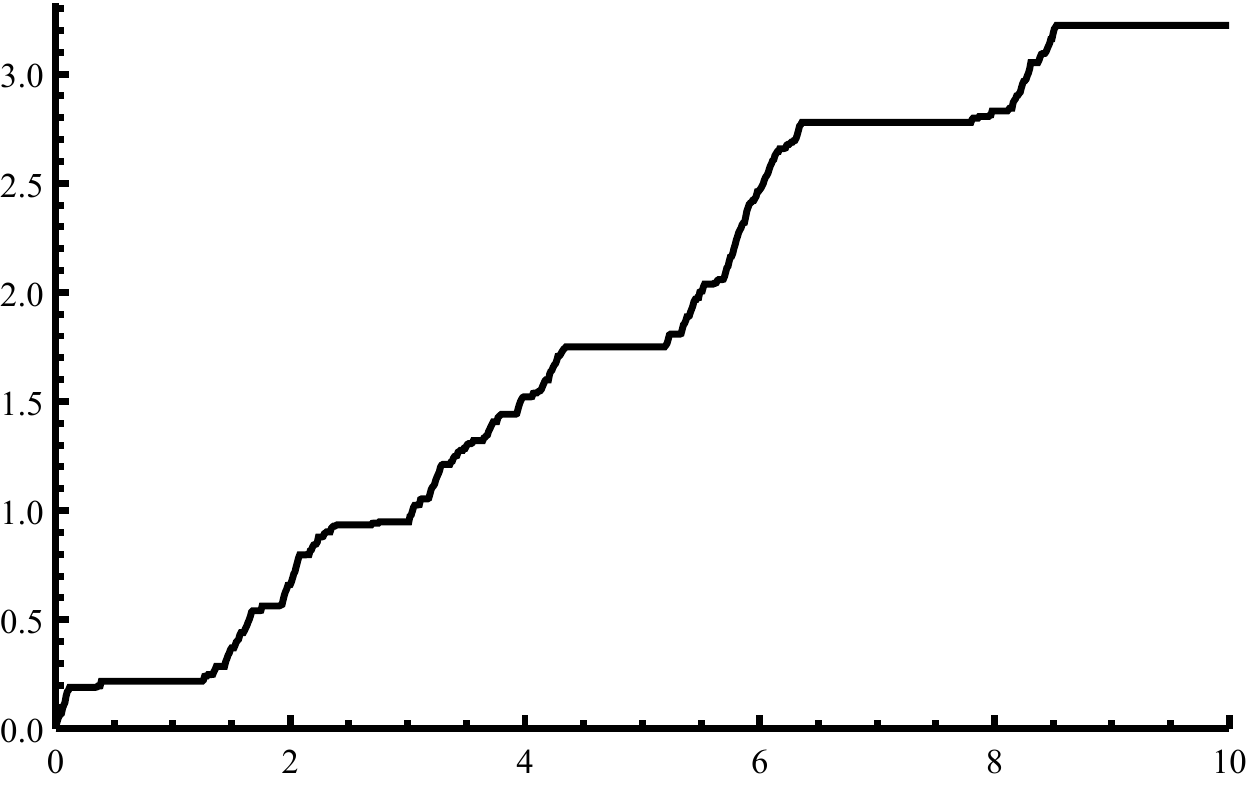}\hspace{1cm}\includegraphics[scale=0.5]{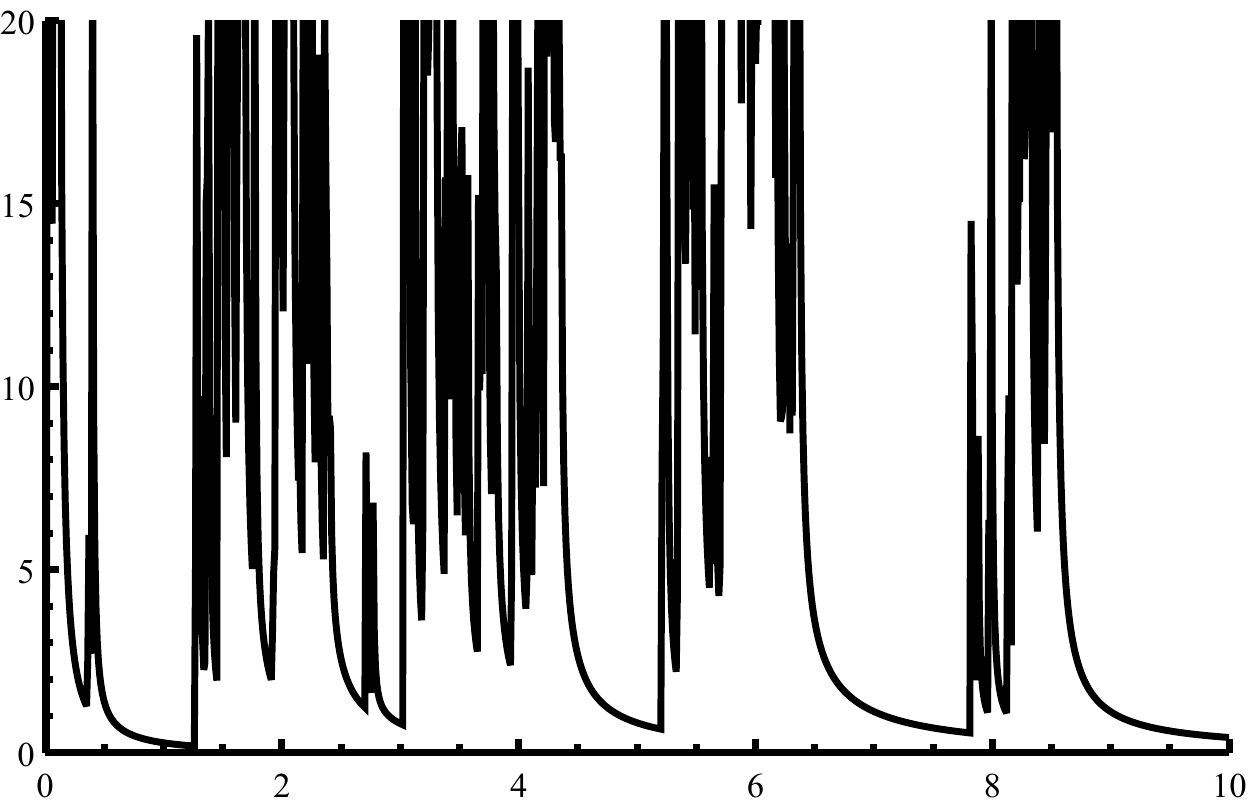}\\
$\alpha=0.75$ and $\beta=-1.5$\\
\end{center}
\caption{Inverse stable subordinators (left) and the corresponding FIISS (right)}
\label{FIISS_fig1}
\end{figure}

The rest of the paper is structured as follows. Main results are formulated in Section \ref{main}. Theorem \ref{xxx} states
that fractionally integrated stable subordinators $Y_{\alpha,\,\beta}$ for $\beta>-\alpha$ are scaling limits
in the Skorokhod space of certain renewal shot noise processes with heavy-tailed `inter-shot' distributions. Since the renewal shot noise processes
are extensively used in diverse areas of applied mathematics, the processes $Y_{\alpha,\,\beta}$, as their limits, may be useful
for heavy-tailed modeling. The paths of $Y_{\alpha,\,\beta}$ for $\beta\leq -\alpha$ are ill-behaved (see Proposition \ref{finite_lemma}). Hence,
the convergence of finite-dimensional distributions provided by Theorem \ref{fin-dim} is the best possible result in this case. The other main results of the paper are concerned with sample path properties of $Y_{\alpha,\,\beta}$. Theorem \ref{holder} is a H\"{o}lder-type result which generalizes \eqref{sam}. Theorem \ref{logar} is the law of iterated logarithm for both small and large times which generalizes \eqref{bert}. In Section \ref{Lamp} we show that $Y_{\alpha,\,\beta}(1)$ has the same distribution as the exponential functional of a killed subordinator by exploiting the Lamperti representation \cite{Lamperti:1972} of semi-stable processes. The main results are proved in Sections \ref{proofxxx}, \ref{proof_holder} and \ref{proof_logar}. The Appendix collects several auxiliary results.

\section{Main results}\label{main}

\subsection{Fractionally integrated inverse stable subordinators as scaling limits of renewal shot noise processes}

Below we shall use the notation introduced in Section \ref{W}.

For a c\`{a}dl\`{a}g function $h$, define $$X(t):=\sum_{k\geq 0}h(t-S_k)\1_{\{S_k\leq t\}}=\int_{[0,\,t]}h(t-y){\rm d}\nu(y),\quad t\geq 0.$$ The process $(X(t))_{t\geq 0}$ is called {\it renewal shot noise process} with response function $h$. There has been an outbreak of recent activity around weak convergence of renewal shot noise processes and their generalizations called {\it random processes with immigration}, see \cite{Alsmeyer+Iksanov+Marynych:2016, Iksanov:2013, Iksanov+Kabluchko+Marynych:2016, Iksanov+Marynych+Meiners:2014, Iksanov+Marynych+Meiners:2016, Iksanov+Marynych+Meiners:2016b}. Both renewal shot noise processes and random processes with immigration are ubiquitous in applied mathematics. Many relevant references can be traced via the last cited articles.
\begin{thm}\label{xxx}
Assume that $\mmp\{\xi>t\}\sim t^{-\alpha}\ell(t)$ for some
$\alpha\in (0,1)$ and some $\ell$ slowly varying at $\infty$.
Let $h:[0,\infty)\to [0,\infty)$ be a right-continuous monotone function that satisfies % $h(0)<\infty$ and
$h(t)\sim t^\beta\widehat{\ell}(t)$ for some $\beta>-\alpha$ and
some $\widehat{\ell}$ slowly varying at $\infty$. Then, as
$t\to\infty$, $$\frac{\mmp\{\xi>t\}}{h(t)}\sum_{k\geq
0}h(ut-S_k)\1_{\{S_k\leq ut\}}\quad \Rightarrow\quad
Y_{\alpha,\,\beta}(u)$$ on $D(0,\infty)$.
\end{thm}
\begin{rem}\label{eventually}
In the case $\beta\geq 0$, Theorem \ref{xxx} was proved in
\cite{Iksanov:2013} under weaker assumptions that
$h:[0,\infty)\to\mr$ is c\`{a}dl\`{a}g, {\it eventually}
nondecreasing and regularly varying. In Section \ref{proofxxx} we
shall show that, in the case $\beta\leq 0$, Theorem \ref{xxx}
holds whenever $h:[0,\infty)\to\mr$ is c\`{a}dl\`{a}g, {\it
eventually} nonincreasing and regularly varying.
\end{rem}

Recall that convergence to a continuous limit in $D(0,\infty)$
equipped with the $J_1$-topology is equivalent to uniform
convergence on $[a,b]$ for any finite positive $a$ and $b$, $a<b$.
Since the limit process in Theorem \ref{xxx} has a.s.\ continuous sample paths %(see \cite{Iksanov+Marynych+Meiners:2014}),
we obtain the following
\begin{cor}
Under the assumptions of Theorem \ref{xxx}, as $t\to\infty$,
$$\frac{\mmp\{\xi>t\}}{h(t)}\underset{u\in [a,\,b]}{\sup}\,\sum_{k\geq
0}h(ut-S_k)\1_{\{S_k\leq ut\}}\quad\dod\quad \underset{u\in
[a,\,b]}{\sup}\,Y_{\alpha,\,\beta}(u)$$ for any finite positive $a$ and $b$,
$a<b$.
\end{cor}

Extremal behavior of shot-noise processes has attracted
considerable attention in the literature, see
\cite{Doney+Obrien:1991, Homble+McCormick:1995,
Hsing+Teugels:1989, Lebedev:2002,
McCormick:1997,McCormick+Seymour:2001}. However, assumptions of
the cited papers were other than ours.

By Proposition \ref{finite_lemma} below, the paths of
$Y_{\alpha,\beta}$ for $\beta\leq -\alpha$ do not belong to the
space $D(0,\infty)$. Although this shows that the classical
functional limit theorem cannot hold, we still have the
convergence of finite-dimensional distributions.
\begin{thm}\label{fin-dim}
Under the same assumptions as in Theorem \ref{xxx} but with
arbitrary $\beta \in\mr$ we have, as $t\to\infty$,
\begin{equation*}%\label{fin-dim-eq}
\frac{\mmp\{\xi>t\}}{h(t)}\bigg(\sum_{k\geq
0}h(u_1t-S_k)\1_{\{S_k\leq u_1t\}},\ldots, \sum_{k\geq
0}h(u_nt-S_k)\1_{\{S_k\leq u_nt\}}\bigg)\dod (Y_{\alpha,\,\beta}(u_1),\ldots, Y_{\alpha,\,\beta}(u_n))
\end{equation*}
for any $n\in\mn$ and any $0<u_1<\ldots<u_n<\infty$.
\end{thm}

Since $t\mapsto \mmp\{\xi>t\}/h(t)$ varies regularly at $\infty$ with index $-\alpha-\beta$, the following statement is immediate.

\begin{assertion}\label{simil}
$Y_{\alpha,\,\beta}$ is self-similar with index $\alpha+\beta$.
\end{assertion}

\subsection{Sample path properties of fractionally integrated inverse stable subordinators}

%The sample paths of $Y_{\alpha,\,\beta}$ are a.s.\ continuous whenever $\beta>-\alpha$ by Proposition 2.18 in \cite{Iksanov+Marynych+Meiners:2013}.
Our first result shows that when $\beta\leq-\alpha$ the sample paths of $Y_{\alpha,\,\beta}$ are rather irregular. % ill-behaved.
%even though for every $u>0$, the random variable $Y_{\alpha,\,\beta}(u)$ is a.s.\ finite.
\begin{assertion}\label{finite_lemma}
Assume that $\beta\leq -\alpha$. Then the random variable
$Y_{\alpha,\,\beta}(u)$ is almost surely finite for any fixed
$u\geq 0$.  However, for every interval $I\subset (0,\infty)$ we
have $\sup_{u\in I}Y_{\alpha,\,\beta}(u)=+\infty$ with positive
probability. Furthermore, with probability one there exist
infinitely many (random) points $u>0$ such that
$Y_{\alpha,\beta}(u)=+\infty$.
\end{assertion}

%\textcolor{red}{%In the complementary case, i.e. when $\alpha+\beta>0$, the paths of $Y_{\alpha,\,\beta}$ are almost surely continuous.
The next theorem is a H\"{o}lder-type result which generalizes
\eqref{sam}.
\begin{thm}\label{holder}
Suppose $\alpha+\beta\in (0,1)$. Then
\begin{equation}\label{Lip2}
\sup_{0\leq v<u\leq
1/2}\,\frac{|Y_{\alpha,\,\beta}(u)-Y_{\alpha,\,\beta}(v)|}{
(u-v)^{\alpha+\beta}|\log
(u-v)|^{1-\alpha}}<\infty\quad\text{a.s.}
\end{equation}
Suppose $\alpha+\beta=1$. Then
\begin{equation}\label{Lip3}
\sup_{0\leq v<u\leq
1/2}\,\frac{Y_{\alpha,\,\beta}(u)-Y_{\alpha,\,\beta}(v)}{(u-v)|\log
(u-v)|^{2-\alpha}}<\infty\quad\text{a.s.}
\end{equation}
In particular, in both cases above $Y_{\alpha,\,\beta}$ is a.s.\
(locally) H\"{o}lder continuous with arbitrary  exponent
$\gamma<\alpha+\beta$. Suppose $\alpha+\beta>1$. Then
\begin{equation}\label{Lip}
\sup_{0\leq v<u\leq
1/2}\,\frac{Y_{\alpha,\,\beta}(u)-Y_{\alpha,\,\beta}(v)}{
u-v}<\infty\quad\text{a.s.}
\end{equation}
which means that $Y_{\alpha,\,\beta}$ is a.s.\ (locally) Lipschitz
continuous.
\end{thm}
\begin{rem}
In the case $\alpha+\beta>1$ the process $Y_{\alpha\,\beta}$ is
actually not only a.s.\ locally Lipschitz continuous, but also
$[\alpha+\beta]$-times continuously differentiable on $[0,\infty)$
a.s. This follows from the equality
$$
Y_{\alpha,\,\beta}(u)=\beta \int_0^u Y_{\alpha,\,\beta-1}(v){\rm
d}v,\quad u\geq 0
$$
which shows that if $Y_{\alpha,\,\beta-1}$ is continuous, then
$Y_{\alpha,\,\beta}$ is continuously differentiable.
\end{rem}

We proceed with the law of iterated logarithm both for small and
large times.
\begin{thm}\label{logar}
Whenever $\beta>-\alpha$ we have
\begin{equation}\label{limit}
\lim\sup \,\frac{Y_{\alpha,\,\beta}(u)}{u^{\alpha+\beta}(\log|\log
u|)^{1-\alpha}}=\frac{1}{\Gamma(1-\alpha)(\alpha+\beta)^\alpha(1-\alpha)^{1-\alpha}}=:c_{\alpha,\,\beta}\quad\text{a.s.}
\end{equation}
and
\begin{equation}\label{limit2}
\lim\inf \,\frac{Y_{\alpha,\,\beta}(u)}{u^{\alpha+\beta}(\log|\log
u|)^{1-\alpha}}=0\quad\text{a.s.}
\end{equation}
both as $u\to 0+$ and $u\to+\infty$. %, where
%$$c_{\alpha,\beta}:=\frac{1}{\Gamma(1-\alpha)(\alpha+\beta)^\alpha(1-\alpha)^{1-\alpha}}.$$
\end{thm}

\section{Distributional properties of the fractionally integrated inverse stable subordinators}\label{Lamp}
Consider a family of processes
$X_\alpha^{(u)}(t):=((u^{1/\alpha}-D_\alpha(t))^\alpha)_{0\leq
t<W_\alpha(u^{1/\alpha})}$ indexed by the initial value $u>0$.
This family forms a semi-stable Markov process of index $1$, i.e.
$$\mmp\{c X_\alpha^{(u)}(t/c)\in \cdot\}=\mmp\{X_\alpha^{(cu)}(t)\in\cdot\}$$ for all $c>0$. Then, according to Theorem 4.1 in \cite{Lamperti:1972}, with $u$ fixed
$$(u^{1/\alpha}-D_\alpha(t))^\alpha=u\exp(-Z_\alpha(\tau(t/u)))\quad\text{for}\quad 0\leq t\leq uI
\quad\text{a.s.}$$ for some killed subordinator
$Z_\alpha:=(Z_\alpha(t))_{t\geq 0}=(Z_\alpha^{(u)}(t))_{t\geq 0}$
where
\begin{equation}\label{aux6}
I:=\int_0^\infty\exp(-Z_\alpha(t)){\rm d}t=u^{-1}\inf\{v:
D_\alpha(v)>u^{1/\alpha}\}=u^{-1}W_\alpha(u^{1/\alpha})
\end{equation}
and $\tau(t):=\inf\{s: \int_0^s\exp(-Z_\alpha(v)){\rm d}v\geq t\}$
for $0\leq t\leq I$ (except in one place, we suppress the
dependence of $Z_\alpha$, $I$ and $\tau(t)$ on $u$ for notational
simplicity). With this at hand
\begin{eqnarray*}
Y_{\alpha,\,\beta}(u^{1/\alpha})&=&\int_0^\infty ((u^{1/\alpha}-D_\alpha(t))^\alpha)^{\beta/\alpha}\1_{\{D_\alpha(t)\leq u^{1/\alpha}\}}{\rm d}t\\&=&
u^{\beta/\alpha}\int_0^{uI}\exp(-(\beta/\alpha)Z_\alpha(\tau(t/u))){\rm d}t=u^{1+\beta/\alpha}\int_0^I\exp(-(\beta/\alpha)Z_\alpha(\tau(t))){\rm d}t\\&=&u^{1+\beta/\alpha}\int_0^\infty\exp(-(1+\beta/\alpha)Z_\alpha(t)){\rm d}t.
\end{eqnarray*}
Replacing $u$ with $u^\alpha$ we infer
\begin{equation}\label{distr}
Y_{\alpha,\,\beta}(u)=u^{\alpha+\beta}\int_0^\infty\exp(-cZ^{(u^{\alpha})}_\alpha(t)){\rm
d}t\quad\text{a.s.}
\end{equation}
where $c:=\alpha^{-1}(\alpha+\beta)$.
%\begin{equation}\label{distr}
%\textcolor{red}{Y_{\alpha,\,\beta}(u) \od u^{\alpha+\beta}\int_0^\infty\exp(-cZ_\alpha(t)){\rm d}t,}
%\end{equation}
%where $c:=\alpha^{-1}(\alpha+\beta)$.
The latter integral is known as an {\it exponential functional of
subordinator}. We shall show that $Z_\alpha$ is a drift-free
killed subordinator with the unit killing rate and the L\'{e}vy
measure
$$\nu_\alpha({\rm d}x)=\frac{e^{-x/\alpha}}{(1-e^{-x/\alpha})^{\alpha+1}}\1_{(0,\,\infty)}(x){\rm d}x.$$ Equivalently,
the Laplace exponent of $Z_\alpha$ equals
$$\Phi_\alpha(s):=-\log\me e^{-sZ_\alpha(1)}=1+\int_{[0,\infty)}(1-e^{-st})\nu_\alpha({\rm
d}x)=\frac{\Gamma(1-\alpha)\Gamma(1+\alpha
s)}{\Gamma(1+\alpha(s-1))},\quad s\geq 0$$ where $\Gamma(\cdot)$
is the gamma function.

It is well known that $W_\alpha(1)$ has a Mittag-Leffler
distribution with parameter $\alpha$. This distribution is
uniquely determined by its moments $$\me
(W_\alpha(1))^n=\frac{n!}{(\Gamma(1-\alpha))^n\Gamma(1+n\alpha)},\quad
n\in\mn.$$ Using \eqref{aux6} along with self-similarity of
$W_\alpha$ we conclude that $I$ has the same Mittag-Leffler
distribution. It follows that the moments of $I$ can be written as
$$\me I^n=\frac{n!}{(\Gamma(1-\alpha))^n\Gamma(1+n\alpha)}=\frac{n!}{\Phi_\alpha(1)\cdot\ldots\cdot\Phi_\alpha(n)},\quad n\in\mn$$
which, by Theorem 2 in \cite{Bertoin+Yor:2005}, implies that the
L\'evy measure of $Z_\alpha$ has the form as stated above.

%which states that thereby showing that $\me I^n=\me (W_\alpha(1))^n$ for $n\in\mn$. Furthermore, if the distribution of $Z_\alpha$ were other than as stated above, then, by the last cited result, the moment sequence of $I$ would be different.

By Lemma \ref{rivero}, $Y_{\alpha,\,\beta}(1)$ has a bounded and nonincreasing density $f_{\alpha,\,\beta}$, say. Since $$\Psi_{\alpha,\,\beta}(s):=-\log \me e^{-scZ_\alpha(1)}=\frac{\Gamma(1-\alpha)\Gamma((\alpha+\beta)s+1)}{\Gamma((\alpha+\beta)s+1-\alpha)},\quad s\geq 0$$ and hence
$$\Psi_{\alpha,\,\beta}(s)\quad \sim\quad \Gamma(1-\alpha)(\alpha+\beta)^\alpha s^\alpha,\quad s\to\infty,$$ another application of Lemma \ref{rivero} allows us to conclude that
\begin{eqnarray}\label{riv}
-\log\mmp\{Y_{\alpha,\,\beta}(1)>x\}&\sim&-\log f_{\alpha,\,\beta}(x)\sim  (1-\alpha)((\alpha+\beta)^\alpha\Gamma(1-\alpha))^{(1-\alpha)^{-1}}x^{(1-\alpha)^{-1}}\notag\\
&=&(x/c_{\alpha,\,\beta})^{(1-\alpha)^{-1}},\quad x\to\infty
\end{eqnarray}
with $c_{\alpha,\,\beta}$ as defined in \eqref{limit}. In particular, for any $\delta_1\in (0,1)$ there exists $c_1=c_1(\delta_1)$ such that
\begin{equation}\label{riv2}
f_{\alpha,\,\beta}(x)\leq c_1\exp\big(-(1-\delta_1)(x/c_{\alpha,\,\beta})^{(1-\alpha)^{-1}}\big)
\end{equation}
for all $x\geq 0$.

\section{Proofs of Theorems \ref{xxx} and \ref{fin-dim}, and Remark \ref{eventually}}\label{proofxxx}

\begin{proof}[Proof of Theorems \ref{xxx} and \ref{fin-dim}]

In the case where $\beta\geq 0$ and $h$ is nondecreasing the result was proved in Theorem 1.1 of \cite{Iksanov:2013}. Therefore, we only investigate the case where $\beta\leq 0$ and $h$ is nonincreasing. In what follows, all unspecified limits are assumed to hold as $t\to\infty$.
%We shall write $\overset{{\rm J_1}}{\Rightarrow}$ to denote weak convergence in the $J_1$-topology on $D(0,\infty)$ or $D[0,\infty)$.

Set $a(t):=\mmp\{\xi>t\}$. First we fix an arbitrary
$\varepsilon\in (0,1)$  and prove that
$$I_\varepsilon(u,t):=\frac{a(t)}{h(t)}\sum_{k\geq 0}h(ut-S_k)\1_{\{S_k\leq \varepsilon ut\}}\quad \Rightarrow \quad \int_{[0,\,\varepsilon
u]}(u-y)^\beta{\rm d}W_\alpha(y)$$ on $D(0,\infty)$. Write
\begin{eqnarray*}
I_\varepsilon(u,t)&=&a(t) \sum_{k\geq 0}\bigg(\frac{h(ut-S_k)}{
h(t)}-(u-t^{-1}S_k)^\beta \bigg)\1_{\{S_k\leq \varepsilon
ut\}}\\&+&a(t) \sum_{k\geq 0}(u-t^{-1}S_k)^\beta \1_{\{S_k\leq
\varepsilon ut\}}\\&=& I_{\varepsilon,\,
1}(u,t)+I_{\varepsilon,\,2}(u,t).
\end{eqnarray*}
We shall show that
\begin{equation}\label{11111}
I_{\varepsilon,\,1}(u,t)\quad \Rightarrow \quad r(u)\quad \text{and}\quad I_{\varepsilon,\,2}(u,t)\quad \Rightarrow \quad \int_{[0,\,\varepsilon
u]}(u-y)^\beta{\rm d}W_\alpha(y)
\end{equation}
on $D(0,\infty)$, where $r(u)=0$ for all $u>0$. Throughout the rest of the proof we use
arbitrary positive and finite $a<b$. Observe that
$$|I_{\varepsilon,\,1}(u,t)|\leq \underset{(1-\varepsilon)u\leq y\leq
u}{\sup}\,\bigg|\frac{h(ty)}{
h(t)}-y^\beta\bigg|a(t)\nu(\varepsilon ut)$$ and thereupon
$$\underset{a\leq u\leq b}{\sup}\,|I_{\varepsilon,\,1}(u,t)|\leq \underset{(1-\varepsilon)a\leq y\leq
b}{\sup}\,\bigg|\frac{h(ty)}{
h(t)}-y^\beta\bigg|a(t)\nu(\varepsilon bt).$$ As a consequence of
the functional limit theorem for $(\nu(t))_{t\geq 0}$ (see
\eqref{flt_nu}), $a(t)\nu(\varepsilon bt)\dod W_\alpha(\varepsilon
b)$. This, combined with the uniform convergence theorem for
regularly varying functions (Theorem 1.2.1 in \cite{BGT}), implies
that the last expression converges to zero in probability thereby
proving the first relation in \eqref{11111}.

Turning to the second relation in \eqref{11111} we observe that\footnote{Below ${\rm d}\nu(ty)$ and ${\rm d}\left(a(t)\nu(ty))\right)$ denote the differential over $y$.}
$$
I_{\varepsilon,\,2}(u,t)=\int_{[0,\,\varepsilon u]}(u-y)^{\beta}{\rm d}\left(a(t)\nu(ty)\right).
$$
Recall from \eqref{flt_nu} that $a(t)\nu(ty) \Rightarrow
W_\alpha(y)$ weakly on $D[0,\infty)$, as $t\to\infty$. Using the
Skorokhod representation theorem, we can pass to versions which
converge a.s.\ in the $J_1$-topology. Since the limit $W_\alpha$
is continuous, the a.s.\ convergence is even locally uniform on
$[0,\infty)$. Applying Lemma \ref{continuity_of_convolution_in_D}
from the Appendix, we obtain the second relation in \eqref{11111}.

An appeal to Theorem 3.1 in \cite{Billingsley:1999} reveals that
the proof of Theorem \ref{fin-dim} is complete if we can show that
for any $\beta\leq 0$ and any {\it fixed} $u>0$
\begin{equation}\label{33333}
\lim_{\varepsilon\to 1-}\int_{[0,\,\varepsilon u]}(u-y)^\beta{\rm
d}W_\alpha(y)\quad = \quad Y_{\alpha,\,\beta}(u)= \int_{[0,\,
u]}(u-y)^{\beta}{\rm d}W_\alpha(y)\quad\text{a.s.}
\end{equation}
and
\begin{equation}\label{bill_cond_fd}
\underset{\varepsilon\to 1-}{\lim}\,\underset{t\to\infty}{\lim\sup}\,\mmp\bigg\{\frac{a(t)}{h(t)}\sum_{k\geq
0}h(ut-S_k)\1_{\{\varepsilon ut<S_k\leq ut\}}>\theta\bigg\}=0
\end{equation}
for all $\theta>0$. Analogously, Theorem \ref{xxx} follows once we
can show that for $\beta\in (-\alpha, 0]$ the following two
statements hold. First, the a.s.\ convergence in \eqref{33333} is
locally uniform on $(0,\infty)$. Second, a uniform analog of
\eqref{bill_cond_fd} holds, namely,
\begin{equation}\label{bill_cond}
\underset{\varepsilon\to
1-}{\lim}\,\underset{t\to\infty}{\lim\sup}\,\mmp\bigg\{\frac{a(t)}{h(t)}
\underset{u\in [a,\,b]}{\sup}\sum_{k\geq
0}h(ut-S_k)\1_{\{\varepsilon ut<S_k\leq ut\}}>\theta\bigg\}=0
\end{equation}
for all $\theta>0$.

To check that \eqref{33333} holds pointwise for any $\beta\leq 0$,
write for fixed $u>0$
\begin{eqnarray*}
0\leq \int_{[0,\, u]}(u-y)^\beta{\rm
d}W_\alpha(y)&-&\int_{[0,\,\varepsilon
u]}(u-y)^\beta{\rm
d}W_\alpha(y)=\int_{[0,\,u]}(u-y)^{\beta}\1_{(\varepsilon u,\,u]}(y){\rm d}W_{\alpha}(y).
\end{eqnarray*}
By the dominated convergence theorem, the right-hand side
converges to $0$ a.s.\ as $\varepsilon \to 1-$ because
$\int_{[0,\,u]}(u-y)^{\beta}{\rm d}W_{\alpha}(y)<\infty$ a.s.\ by
Proposition \ref{finite_lemma}.

The probability on the left-hand side of \eqref{bill_cond_fd} is
bounded from above by
\begin{eqnarray*}
\mmp\{\nu(ut)-\nu(\varepsilon
ut)>0\}&=&\mmp\{\nu(ut)-\nu(\varepsilon ut)\geq
1\}=\mmp\{ut-S_{\nu(ut)-1}<(1-\varepsilon)ut\}.
\end{eqnarray*}
By a well-known Dynkin-Lamperti result (see Theorem 8.6.3 in
\cite{BGT})
$$
t^{-1}(t-S_{\nu(t)-1})\quad \dod \quad \eta_\alpha
$$
where $\eta_\alpha$ has a beta distribution with parameters
$1-\alpha$ and $\alpha$, i.e., $$\mmp\{\eta_\alpha\in {\rm
d}x\}=\pi^{-1}\sin
(\pi\alpha)x^{-\alpha}(1-x)^{\alpha-1}\1_{(0,1)}(x){\rm d}x.$$ % has
%the generalized arcsine law with parameter $\alpha$, thereby
This entails
$$
\underset{\varepsilon\to
1-}{\lim}\,\underset{t\to\infty}{\lim\sup}\,\mmp\{\nu(ut)-\nu(\varepsilon
ut)>0\}= \underset{\varepsilon\to
1-}{\lim}\,\mmp\{\eta_\alpha<1-\varepsilon\}=0
$$
thereby proving \eqref{bill_cond_fd}. The proof of Theorem
\ref{fin-dim} is complete.

Now we turn to the proof of Theorem \ref{xxx}. In particular,
$\beta\in (-\alpha, 0]$ is the standing assumption in what
follows.

The right-hand side of \eqref{33333} is a.s.\ continuous (see
point (II) in Section \ref{properties}). Further, it can be
checked that the left-hand side of \eqref{33333} is a.s.\
continuous, too. Since it is also monotone in $\varepsilon$ we can
invoke Dini's theorem to conclude that the a.s.\ convergence in
\eqref{33333} is locally uniform on $(0,\infty)$.
% . with probability one, the functions on both sides of
%\eqref{33333} are continuous and, moreover, the left-hand side is
%monotone in $\varepsilon$. Therefore, by Dini's theorem, the
%convergence in \eqref{33333} is locally uniform on $(0,\infty)$
%with probability one.

To check \eqref{bill_cond}, we need the following proposition to
be proved in the Appendix.

\begin{assertion}\label{renewal_modulus_of_continuity}
Fix $T>0$ and set $A_t:=\{(u,v):0\leq v<u\leq T,u-v\geq 1/t\}$ for
$t>0$. If $a(t)=\mmp\{\xi>t\}\sim t^{-\alpha}\ell(t)$ for some
$\alpha\in (0,1)$ and some  $\ell$ slowly varying at $\infty$,
then, for any $\delta\in(0,\alpha)$,
\begin{equation*}
\lim_{x\to\infty}\limsup_{t\to\infty}\mmp\left\{\sup_{(u,v)\in A_t}\frac{a(t)(\nu(ut)-\nu(vt))}{(u-v)^{\alpha-\delta}} > x\right\}=0.
\end{equation*}
\end{assertion}

%Returning to the proof of \eqref{bill_cond},
Fix now $\Delta\in(0,(\alpha+\beta)/2)$ and note that by Potter's
bound for regularly varying functions (Theorem 1.5.6 in
\cite{BGT}) there exists $c>1$ such that
\begin{equation*}\label{potters_bound}
\frac{h(t(u-y))}{h(t)}\leq 2(u-y)^{\beta-\Delta}
\end{equation*}
for all $t$, $u$ and $y$ such that $t(u-y)\geq c$ and $u-y\leq 1$. With this at hand, we have for $t$ large enough, $u\in [a,b]$ and $\varepsilon>0$ such that $(1-\varepsilon)b\leq 1$
\begin{eqnarray*}
&&\frac{a(t)}{h(t)}\sum_{k\geq
0}h(ut-S_k)\1_{\{\varepsilon ut<S_k\leq ut\}}\\&=&\frac{a(t)}{h(t)}\int_{(\varepsilon u,\,u-c/t]}h(t(u-y)){\rm d}\nu(ty)+\frac{a(t)}{h(t)}\int_{(u-c/t,\,u]}h(t(u-y)){\rm d}\nu(ty)\\
&\leq& 2a(t)\int_{(\varepsilon u,\,u-c/t]}(u-y)^{\beta-\Delta}{\rm d}\nu(ty)+\frac{a(t)}{h(t)}h(0)(\nu(tu)-\nu(tu-c))\\&=&
2(-\beta+\Delta)\int_{\varepsilon u}^{u-c/t}a(t)(\nu(tu)-\nu(ty))(u-y)^{\beta-\Delta-1}{\rm d}y\\&+&2u^{\beta-\Delta}(1-\varepsilon)^{\beta-\Delta}a(t)(\nu(tu)-\nu(\varepsilon tu ))\\&+& \bigg(\frac{a(t)}{h(t)}h(0)-2c^{\beta-\Delta}t^{-\beta+\Delta}a(t)\bigg)(\nu(tu)-\nu(tu-c)).
\end{eqnarray*}
Since $t\mapsto a(t)/h(t)$ and $t\mapsto t^{-\beta+\Delta}a(t)$ are regularly varying of negative indices $-\alpha-\beta$ and $-\alpha-\beta+\Delta$, respectively, we have
$$
\bigg(\frac{a(t)}{h(t)}h(0)-2c^{\beta-\Delta}t^{-\beta+\Delta}a(t)\bigg)\sup_{u\in
[a,\,b] }(\nu(tu)-\nu(tu-c))\quad\tp\quad 0$$ by Lemma
\ref{iks13}. Further,
$$
\sup_{u\in[a,b]}u^{\beta-\Delta}(1-\varepsilon)^{\beta-\Delta}a(t)(\nu(tu)-\nu(\varepsilon tu ))\leq a^{\beta-\Delta}(1-\varepsilon)^{\beta-\Delta}a(t)\sup_{u\in[a,\,b]}(\nu(tu)-\nu(\varepsilon tu ))
$$
and $$a(t)\sup_{u\in[a,\,b]}(\nu(tu)-\nu(\varepsilon tu
))\quad\dod\quad \underset{u\in
[a,\,b]}{\sup}\,(W_\alpha(u)-W_\alpha(\varepsilon u))$$ in view of
\eqref{flt_nu} and the continuous mapping theorem. Therefore,
\begin{eqnarray*}
&&\hspace{-2cm}\lim_{\varepsilon\to 1-0}\limsup_{t\to\infty}\mmp\left\{\sup_{u\in[a,b]}u^{\beta-\Delta}(1-\varepsilon)^{\beta-\Delta}a(t)(\nu(tu)-\nu(\varepsilon tu ))>\theta\right\}\\
&\leq&\lim_{\varepsilon\to 1-0}\mmp\left\{a^{\beta-\Delta}(1-\varepsilon)^{\beta-\Delta}\sup_{u\in[a,b]}(W_{\alpha}(u)-W_{\alpha}(u\varepsilon ))>\theta\right\}\\
&=&\lim_{\varepsilon\to 1-0}\mmp\left\{a^{\beta-\Delta}(1-\varepsilon)^{\beta-\Delta}(2b)^{\alpha}\sup_{u\in[a/2b,1/2]}(W_{\alpha}(u)-W_{\alpha}(u\varepsilon ))>\theta\right\}=0
\end{eqnarray*}
for all $\theta>0$, where the penultimate equality is a
consequence of self-similarity of $W_\alpha$, and the last
equality is implied by \eqref{sam} and the choice of $\Delta$.

%Hence, \eqref{bill_cond} follows from
%\begin{equation}\label{bill_cond1}
%\underset{\varepsilon\uparrow 1}{\lim}\,\underset{t\to\infty}{\lim\sup}\,\mmp\bigg\{a(t)
%\underset{u\in [a,\,b]}{\sup}\int_{(\varepsilon u,\,u-c/t]}(u-y)^{-\gamma-\Delta}{\rm d}\nu(ty)>\rho\bigg\}=0,
%\end{equation}
%for all $\rho>0$. %To check \eqref{bill_cond1}, use integration by parts:
%\begin{eqnarray*}
%a(t)\int_{(\varepsilon u,\,u-c/t]}(u-y)^{-\gamma-\Delta}{\rm d}\nu(ty)&=&-a(t)\int_{(\varepsilon u,\,u-c/t]}(u-y)^{-\gamma-\Delta}{\rm d}(\nu(tu)-\nu(ty))\\
%&&\hspace{-6cm}=-c^{-\gamma-\Delta}t^{\gamma+\Delta}a(t)(\nu(tu)-\nu(tu-c))+u^{-\gamma-\Delta}(1-\varepsilon)^{-\gamma-\Delta}a(t)(\nu(tu)-\nu(\varepsilon tu ))\\
%&&\hspace{-6cm}+(\gamma+\Delta)\int_{\varepsilon u}^{u-c/t}a(t)(\nu(tu)-\nu(ty))(u-y)^{-\gamma-\Delta-1}{\rm d}y.
%\end{eqnarray*}
%Using Lemma A.1 in \cite{Iksanov:2013} and the regular variation of $t\mapsto t^{\gamma+\Delta}a(t)$ with negative index we infer
%$$
%\sup_{u\in[a,b]} t^{\gamma+\Delta}a(t)(\nu(tu)-\nu(tu-c))\tp 0,\quad t\to\infty.
%$$

%as $\varepsilon\downarrow 1$, by \eqref{Haw} and the choice of $\Delta$.
Hence, \eqref{bill_cond} follows if we can show that
\begin{equation}\label{bill_cond1}
\underset{\varepsilon\to
1-}{\lim}\,\underset{t\to\infty}{\lim\sup}\,\mmp\left\{\sup_{u\in[0,T]}\int_{\varepsilon
u}^{u-c/t}a(t)(\nu(tu)-\nu(ty))(u-y)^{\beta-\Delta-1}{\rm
d}y>\theta\right\}=0
\end{equation}
for all $\theta>0$ and all $T>0$. With
$0<\delta<\alpha+\beta-\Delta$ the following inequality holds:
\begin{eqnarray*}
&&\mmp\left\{\sup_{u\in[0,T]}\int_{\varepsilon u}^{u-c/t}a(t)(\nu(tu)-\nu(ty))(u-y)^{\beta-\Delta-1}{\rm d}y>\theta\right\}\\
&&=\mmp\left\{\cdots,\sup_{(u,v)\in A_t}\frac{a(t)(\nu(ut)-\nu(vt))}{(u-v)^{\alpha-\delta}}> x\right\}+
\mmp\left\{\cdots,\sup_{(u,v)\in A_t}\frac{a(t)(\nu(ut)-\nu(vt))}{(u-v)^{\alpha-\delta}}\leq x\right\}\\
&&\leq \mmp\left\{\sup_{(u,v)\in A_t}\frac{a(t)(\nu(ut)-\nu(vt))}{(u-v)^{\alpha-\delta}}>x\right\}+
\mmp\left\{\sup_{u\in[0,T]}\int_{\varepsilon u}^u(u-y)^{\alpha+\beta-\Delta-\delta-1}{\rm d}y>\delta/x\right\}\\
&&=\mmp\left\{\sup_{(u,v)\in
A_t}\frac{a(t)(\nu(ut)-\nu(vt))}{(u-v)^{\alpha-\delta}}>x\right\}+\mmp\left\{\int_0^{(1-\varepsilon)T}y^{\alpha+\beta-\Delta-\delta-1}{\rm
d}y>\delta/x\right\}
\end{eqnarray*}
for $x>0$. Sending $t\to\infty$ and then $\varepsilon\to 1-$ and
$x\to\infty$ and using Proposition
\ref{renewal_modulus_of_continuity} for the first summand on the
right-hand side finishes the proof of \eqref{bill_cond1}. The
proof of Theorem \ref{xxx} is complete.
\end{proof}

\begin{proof}[Proof of Remark \ref{eventually}]
Let $h:[0,\infty)\to\mr$ be a c\`{a}dl\`{a}g function which is nonincreasing on $[d,\infty)$ for some $d>0$ and satisfies $h(t)\sim t^\beta\widehat{\ell}(t)$ for some $\beta\in (-\alpha, 0]$ as $t\to\infty$. Further, let $h^\ast: [0,\infty)\to[0,\infty)$ be any right-continuous nonincreasing function such that $h^\ast(t)=h(t)$ for $t\geq d$.

Then, for any positive and finite $a<b$,
\begin{eqnarray*}
&&\underset{u\in [a,\,b]}{\sup}\,\bigg|\sum_{k\geq 0}h(ut-S_k)\1_{\{S_k\leq ut\}}-\sum_{k\geq 0}h^\ast(ut-S_k)\1_{\{S_k\leq ut\}}\bigg|\\&\leq&
\underset{u\in [a,\,b]}{\sup}\,\sum_{k\geq 0}\big|h(ut-S_k)-h^\ast(ut-S_k)\big|\1_{\{ut-d<S_k\leq ut\}}\\&\leq&\underset{y\in [0,\,d]}{\sup}\,|h(y)-h^\ast(y)|\underset{u\in [a,\,b]}{\sup}\,(\nu(ut)-\nu(ut-d)).
\end{eqnarray*}
The normalization $\mmp\{\xi>t\}/h(t)$ used in Theorem \ref{xxx} is regularly varying of index $-\alpha-\beta$ which is negative in the present situation. This implies that, as $t\to\infty$, the right-hand side of the last centered formula multiplied by $\mmp\{\xi>t\}/h(t)$ converges to zero in probability by Lemma \ref{iks13} which justifies Remark \ref{eventually}.
\end{proof}

\section{Proofs of Proposition \ref{finite_lemma} and Theorem \ref{holder}}\label{proof_holder}
\begin{proof}[Proof of Proposition \ref{finite_lemma}]
Let $\mathcal{R}$ be the range of subordinator $D_\alpha$ defined
by
$$
\mathcal{R}:=\{t > 0: \text{ there exists } y>0 \text{ such that } D_{\alpha}(y)=t\}.
$$
If $u\notin\mathcal{R}$, then
$$Y_{\alpha,\,\beta}(u)=\int_{[0,\,u]}(u-y)^{\beta}{\rm d}W_{\alpha}(y)=\int_{[0,\,D_{\alpha}(W_{\alpha}(u)-)]}(u-y)^{\beta}{\rm d}W_{\alpha}(y)$$
because $W_\alpha$ takes a constant value on
$(D_{\alpha}(W_{\alpha}(u)-),u]$ and
$D_{\alpha}(W_{\alpha}(u)-)<u$. This shows that
$Y_{\alpha,\,\beta}(u)<\infty$ for all $u\notin\mathcal{R}$. For
each fixed $u>0$ we have $\mmp\{u\in\mathcal{R}\}=0$ (see
Propostion 1.9 in \cite{Bertoin:1999}) whence
$Y_{\alpha,\,\beta}(u)<\infty$ a.s.

The proof of unboundedness goes along the same lines as that of
Proposition 2.7 in \cite{Iksanov+Marynych+Meiners:2016}. Recall
that $\beta<-\alpha<0$ and that $I$ is a fixed interval $[c,d]$,
say. Pick arbitrary positive $a<b$
%such that $da^{1/\alpha}>cb^{1/\alpha}$
and note that
$$
\mmp\{[D_{\alpha}(a),\,D_{\alpha}(b)]\subset [c,d]\} =\mmp\{ c\leq D_{\alpha}(a)< D_{\alpha}(b)\leq d\} > 0.
%=\mmp\{ ca^{-1/\alpha}\leq D_{\alpha}(1)\leq db^{-1/\alpha}\}>0.
$$
Let us now check that
\begin{equation}\label{unbounded_on_a_b}
\sup_{u\in
[D_{\alpha}(a),\,D_{\alpha}(b)]}Y_{\alpha,\,\beta}(u)=\infty\quad\text{a.s.},
\end{equation}
thereby showing that $\sup_{u\in I}Y_{\alpha,\,\beta}(u)=+\infty$
%\begin{equation*}
%\sup_{u\in I}Y_{\alpha,\,\beta}(u)=\infty
%\end{equation*}
with positive probability.

According to Theorem 2 in \cite{Fristedt:1979}, there
exists an event $\Omega^\prime$ with $\mmp\{\Omega^\prime\}=1$ such
that for any $\omega\in \Omega^\prime$
\begin{equation}\label{fristedt1}  \limsup_{y\to s-}\frac{D_\alpha (s,\omega)-D_{\alpha}(y,\omega)}{(s-y)^{1/\alpha}}\leq r
\end{equation}
for some deterministic constant $r\in (0,\infty)$ and some
$s:=s(\omega)\in [a,\,b]$. Fix any $\omega\in \Omega^\prime$. There
exists $s_1:=s_1(\omega)$ such that
$$  \big(D_{\alpha}(s,\omega)-D_{\alpha}(y,\omega)\big)^{\beta} \geq (s-y)^{\beta/\alpha} r^{\beta}/2    $$
whenever $y\in (s_1, s)$. Set $u:=u(\omega)=D_{\alpha}(s,\omega)$
and write
\begin{eqnarray*}
Y_{\alpha,\,\beta}(u)&=&\int_{[0,\,u(\omega)]}(u(\omega)-y)^{\beta}
\, {\rm d} W_\alpha(y,\omega)=
\int_{[0,\,D_\alpha(s,\,\omega)]}(D_{\alpha}(s,\omega)-y)^{\beta} \, {\rm d}W_\alpha(y,\omega)  \\
& = &
\int_0^s\big(D_{\alpha}(s,\omega)-D_\alpha(y,\omega)\big)^{\beta}
\, {\rm d}y
\geq \int_{s_1}^s \big(D_\alpha(s,\omega)-D_\alpha(y,\omega)\big)^{\beta} \, {\rm d}y    \\
& \geq & 2^{-1} r^{\beta} \int_{s_1}^s (s-y)^{\beta/\alpha} \,
{\rm d}y ~=~ +\infty.
\end{eqnarray*}
Since $u(\omega) \in [D_{\alpha}(a),\,D_{\alpha}(b)]$ for all $\omega\in \Omega^\prime$, we obtain \eqref{unbounded_on_a_b}.

Clearly, there are infinitely many positive $s$ such that
\eqref{fristedt1} holds. Hence, % whence with probability one,
$Y_{\alpha,\,\beta}(u)=+\infty$ for infinitely many $u>0$ a.s. The
proof of Proposition \ref{finite_lemma} is complete.
\end{proof}

Our proof of Theorem \ref{holder} will be pathwise, hence
deterministic, in the following sense. In view of \eqref{sam},
there exists an event $\Omega_1$ with $\mmp\{\Omega_1\}=1$ such
that $M=M(\omega)<\infty$ for all $\omega\in \Omega_1$. Below we
shall work with fixed but arbitrary $\omega\in \Omega_1$.

From the very beginning we want to stress that local H\"{o}lder
continuity follows immediately from Theorem 3.1 on p.~53 and Lemma
13.1 on p.~239 in \cite{Samko+Kilbas+Marichev:1993} when $\beta>0$
and $-\alpha<\beta<0$, respectively. However, proving \eqref{Lip2} and
\eqref{Lip3} requires additional efforts.
\begin{proof}[Proof of Theorem \ref{holder}]
Observe that
\begin{equation}\label{sam2}
W_\alpha(x)-W_\alpha(y)\leq M(x-y)^\alpha|\log(x-y)|^{1-\alpha}
\end{equation}
whenever $-\infty<y<x\leq 1/2$. This is trivial when
$x\leq 0$ and is a consequence of \eqref{sam} when $y\geq 0$.
Assume that $y\leq 0<x$. Then
$W_\alpha(x)-W_\alpha(y)=W_\alpha(x-y+y)\leq W_\alpha(x-y)\leq
M(x-y)^\alpha|\log(x-y)|^{1-\alpha}$, where the penultimate
inequality is implied by monotonicity, and the last follows from
\eqref{sam}.

When $\beta=0$, inequality \eqref{Lip2} reduces to \eqref{sam}. We
shall treat the other cases separately.

\noindent {\sc Case $-\alpha<\beta<0$}. Let $1/2\geq u>v>0$. Using \eqref{repr2} we have
\begin{eqnarray*}
&&|\beta|^{-1}|Y_{\alpha,\,\beta}(u)-Y_{\alpha,\,\beta}(v)|\\&=&\bigg|\int_0^\infty
(W_\alpha(u)-W_\alpha(u-y)-W_\alpha(v)+W_\alpha(v-y))y^{\beta-1}
{\rm d}y\bigg|\\&\leq&
\int_0^{u-v}(W_\alpha(u)-W_\alpha(u-y))y^{\beta-1}{\rm
d}y+\int_0^{u-v}(W_\alpha(v)-W_\alpha(v-y))y^{\beta-1}{\rm
d}y\\&+&\int_{u-v}^\infty(W_\alpha(u)-W_\alpha(v))y^{\beta-1}{\rm
d}y+\int_{u-v}^\infty (W_\alpha(u-y)-W_\alpha(v-y))y^{\beta-1}{\rm
d}y\\&\leq& 2M\bigg(\int_0^{u-v}y^{\alpha+\beta-1}|\log
y|^{1-\alpha}{\rm
d}y+(u-v)^\alpha|\log(u-v)|^{1-\alpha}\int_{u-v}^\infty
y^{\beta-1}{\rm
d}y\bigg)\\&=&2M\bigg(\int_0^{u-v}y^{\alpha+\beta-1}|\log
y|^{1-\alpha}{\rm
d}y+|\beta|^{-1}(u-v)^{\alpha+\beta}|\log(u-v)|^{1-\alpha}\bigg)
\end{eqnarray*}
having utilized \eqref{sam2} for the last inequality.
Further,
\begin{eqnarray}
\int_0^{u-v}y^{\alpha+\beta-1}|\log y|^{1-\alpha}{\rm
d}y&=&(u-v)^{\alpha+\beta}\int_0^1
t^{\alpha+\beta-1}|\log(u-v)+\log t|^{1-\alpha}{\rm d}t\notag
\\&\leq& (u-v)^{\alpha+\beta}|\log(u-v)|^{1-\alpha}\int_0^1
t^{\alpha+\beta-1}{\rm d}t\notag \\&+&(u-v)^{\alpha+\beta}\int_0^1
t^{\alpha+\beta-1}|\log t|^{1-\alpha}{\rm d}t\notag \\&\leq&
\bigg(\frac{1}{\alpha+\beta}+ \frac{\int_0^1 t^{\alpha+\beta-1}|\log
t|^{1-\alpha}{\rm d}t}{ (\log
2)^{1-\alpha}}\bigg)(u-v)^{\alpha+\beta}|\log(u-v)|^{1-\alpha}\notag
\\&=:&
\kappa_{\alpha,\,\beta}(u-v)^{\alpha+\beta}|\log(u-v)|^{1-\alpha}.\label{inter2}
\end{eqnarray}
Thus, we have proved that
\begin{equation}\label{aux1}
|Y_{\alpha,\,\beta}(u)-Y_{\alpha,\,\beta}(v)|\leq
2M(|\beta|\kappa_{\alpha,\beta} + 1)
(u-v)^{\alpha+\beta}|\log(u-v)|^{1-\alpha}
\end{equation}
whenever $1/2\geq u>v>0$.

The proof for the case $1/2\geq u>v=0$ proceeds
similarly but simpler and starts with the equality
$$Y_{\alpha,\,\beta}(u)-Y_{\alpha,\,\beta}(0)=Y_{\alpha,\,\beta}(u)=u^\beta W_\alpha(u)+
|\beta|\int_0^u (W_\alpha(u)-W_\alpha(u-y))y^{\beta-1} {\rm d}y.$$ The resulting estimate is
\begin{equation}\label{aux2}
Y_{\alpha,\,\beta}(u)\leq
M(|\beta|\kappa_{\alpha,\beta}+1)u^{\alpha+\beta}|\log
u|^{1-\alpha}
\end{equation}
whenever $1/2\geq u>0$. Combining \eqref{aux1} and \eqref{aux2} proves \eqref{Lip2}.

\noindent {\sc Case $\beta>0$}. Let $1/2\geq u>v\geq 0$. Then
$Y_{\alpha,\,\beta}(u)\geq Y_{\alpha,\,\beta}(v)$. Setting
$I(u,v):=\int_{[0,\,v]}((u-y)^\beta-(v-y)^\beta){\rm
d}W_\alpha(y)$ we obtain
\begin{eqnarray}\label{inter}
Y_{\alpha\,\,\beta}(u)-Y_{\alpha,\,\beta}(v)&=&\int_{[0,\,v]}((u-y)^\beta-(v-y)^\beta){\rm
d}W_\alpha(y) +\int_{(v,\,u]}(u-y)^\beta{\rm d}W_\alpha(y)\notag \\
&\leq& I(u,v)+ (u-v)^\beta (W_\alpha(u)-W_\alpha(v))\notag\\&\leq&
I(u,v)+M (u-v)^{\alpha+\beta}|\log (u-v)|^{1-\alpha}
\end{eqnarray}
where the last inequality is a consequence of \eqref{sam2}.

\noindent {\sc Subcase $\beta\geq 1$}. We have
$(u-y)^\beta-(v-y)^\beta\leq \beta(u-y)^{\beta-1}(u-v)\leq
\beta(u-v)$ by the mean value theorem for differentiable
functions. Hence $I(u,v)\leq \beta W_\alpha(1/2)(u-v)$. This, together with \eqref{inter} and the inequality
\begin{equation}\label{ineq}
x^{\alpha+\beta}|\log x|^{1-\alpha}\leq cx
\end{equation}
which holds for $x\in (0,1/2]$ and some $c>0$, proves \eqref{Lip}.

\noindent {\sc Subcase $\alpha+\beta>1$ and $0<\beta<1$}. An appeal
to the case $-\alpha<\beta<0$ that we have already settled allows us to conclude that
$Y_{\alpha,\,\beta-1}$ is a.s.\ continuous on $[0,\,1/2]$ which
implies
$$\sup_{v\in [0,\,1/2]}\,Y_{\alpha,\,\beta-1}(v)<\infty\quad\text{a.s.}$$
Another application of the mean value theorem yields
$(u-y)^\beta-(v-y)^\beta\leq \beta(v-y)^{\beta-1}(u-v)$ and
thereupon $$I(u,v)\leq \beta(u-v)\int_{[0,v]}(v-y)^{\beta-1}{\rm
d}W_\alpha(y)\leq \beta  (u-v) \sup_{v\in
[0,\,1/2]}\,Y_{\alpha,\,\beta-1}(v).$$ Recalling
\eqref{inter} and \eqref{ineq}, we arrive at \eqref{Lip}.

\noindent {\sc Subcase $\alpha+\beta\leq 1$ and $\beta>0$}. We use
\eqref{repr1} together with a decomposition given on p.~54 in
\cite{Samko+Kilbas+Marichev:1993}:
\begin{eqnarray*}
Y_{\alpha,\,\beta}(u)-Y_{\alpha,\,\beta}(v)&=&W_\alpha(v)(u^\beta-v^\beta)-\beta\int_0^{u-v}(W_\alpha(v)-W_\alpha(u-y))
y^{\beta-1} {\rm d}y\\&+&\beta\int_0^v(W_\alpha(v)-W_\alpha(v-y))
(y^{\beta-1}-(y+u-v)^{\beta-1}){\rm d}y\leq  I_1+I_2
\end{eqnarray*}
where $$I_1:=W_\alpha(v)(u^\beta-v^\beta)\quad\text{and}\quad
I_2:=\beta\int_0^v(W_\alpha(v)-W_\alpha(v-y))
(y^{\beta-1}-(y+u-v)^{\beta-1}){\rm d}y.$$ We first obtain a
preliminary estimate for $I_2$. Using \eqref{sam2}, changing the
variable and then using the subadditivity of $x\to x^{1-\alpha}$
we obtain
\begin{eqnarray*}
I_2&\leq& \beta M\int_0^v y^\alpha|\log
y|^{1-\alpha}(y^{\beta-1}-(y+u-v)^{\beta-1}){\rm d}y\\&=&\beta M
(u-v)^{\alpha+\beta}\int_0^{v/(u-v)}t^\alpha|\log(u-v)+\log
t|^{1-\alpha}(t^{\beta-1}-(t+1)^{\beta-1}){\rm d}t\\&\leq& \beta M
(u-v)^{\alpha+\beta}|\log
(u-v)|^{1-\alpha}\int_0^{v/(u-v)}t^\alpha(t^{\beta-1}-(t+1)^{\beta-1}){\rm
d}t\\&+&\beta M (u-v)^{\alpha+\beta}\int_0^{v/(u-v)}t^\alpha|\log
t|^{1-\alpha}(t^{\beta-1}-(t+1)^{\beta-1}){\rm d}t.
\end{eqnarray*}
Further we distinguish two cases.

\noindent Let $v\leq u-v$. Then
\begin{eqnarray*}
I_2&\leq& \beta M (u-v)^{\alpha+\beta}|\log
(u-v)|^{1-\alpha}\int_0^1 t^{\alpha+\beta-1}{\rm d}t\\&+&\beta M
(u-v)^{\alpha+\beta}\int_0^1 t^{\alpha+\beta-1}|\log
t|^{1-\alpha}{\rm d}t\leq \beta M
\kappa_{\alpha,\,\beta}(u-v)^{\alpha+\beta}|\log(u-v)|^{1-\alpha}
\end{eqnarray*}
with $\kappa_{\alpha,\,\beta}$ defined in \eqref{inter2}. As for $I_1$,
we infer $$I_1\leq W_\alpha(u-v)(u-v)^\beta\leq
M(u-v)^{\alpha+\beta}|\log(u-v)|^{1-\alpha}$$ having utilized
monotonicity of $W_\alpha$ and subadditivity of $x\mapsto x^\beta$
(observe that $\beta\in (0,1)$) for the first inequality and
\eqref{sam2} for the second.

\noindent Let $v>u-v$. Using the inequality
$x^{\beta-1}-(x+1)^{\beta-1}\leq (1-\beta)x^{\beta-2}$, $x>0$, we
conclude that
\begin{eqnarray*}
I_2&\leq& \beta M (u-v)^{\alpha+\beta}|\log
(u-v)|^{1-\alpha}\bigg(\int_0^1 t^{\alpha+\beta-1}{\rm
d}t+(1-\beta)\int_1^\infty t^{\alpha+\beta-2}{\rm
d}t\bigg)\\&+&\beta M (u-v)^{\alpha+\beta}\bigg(\int_0^1
t^{\alpha+\beta-1}|\log t|^{1-\alpha}{\rm
d}t+(1-\beta)\int_1^\infty t^{\alpha+\beta-2}(\log
t)^{1-\alpha}{\rm d}t\bigg)\\&\leq& \zeta_{\alpha,\,\beta}\,(u-v)^{\alpha+\beta}|\log(u-v)|^{1-\alpha}
\end{eqnarray*}
provided that $\alpha+\beta<1$. Here and in the next centered
formula $\zeta_{\alpha,\,\beta}$ denotes an a.s.\ finite random
variable whose value is of no importance. If $\alpha+\beta=1$, we
have
\begin{eqnarray*}
I_2&\leq& \beta M (u-v)|\log
(u-v)|^{1-\alpha}\bigg(1+(1-\beta)\int_1^{v/(u-v)} t^{-1}{\rm
d}t\bigg)\\&+&\beta M (u-v)\bigg(\int_0^1 |\log t|^{1-\alpha}{\rm
d}t+(1-\beta)\int_1^{v/(u-v)}t^{-1}(\log t)^{1-\alpha}{\rm
d}t\bigg)\\&\leq& \beta M (u-v)|\log
(u-v)|^{1-\alpha}(1+(1-\beta)|\log (u-v)|)\\&+& \beta M
(u-v)\bigg(\int_0^1 |\log t|^{1-\alpha}{\rm
d}t+(1-\beta)(2-\alpha)^{-1}|\log
(u-v)|^{2-\alpha}\bigg)\\&\leq& \zeta_{\alpha,\,\beta}\,(u-v)|\log(u-v)|^{2-\alpha}.
\end{eqnarray*}
Finally we use \eqref{sam2} and $(1+x)^\beta-1\leq \beta x$, $x\geq 0$ to obtain
\begin{eqnarray*}
I_1&\leq& M v^{\alpha+\beta}|\log v|^{1-\alpha}((1+(u-v)/v)^\beta-1)\leq \beta M
v^{\alpha+\beta-1}|\log v|^{1-\alpha}(u-v)\\&\leq& \beta M
(u-v)^{\alpha+\beta}|\log(u-v)|^{1-\alpha}.
\end{eqnarray*}
The proof of Theorem \ref{holder} is complete.
\end{proof}

\section{Proof of Theorem \ref{logar}}\label{proof_logar}

Since $Y_{\alpha,\,\beta}$ is self-similar with index $\alpha+\beta$ (see Proposition \ref{simil}) we conclude that
\begin{equation*}
\frac{Y_{\alpha,\,\beta}(u)}{u^{\alpha+\beta}(\log|\log
u|)^{1-\alpha}}\tp 0
\end{equation*}
as $u\to 0+$ or $u\to+\infty$. Taking an appropriate sequence we
arrive at \eqref{limit2}.

Turning to the upper limit we first prove that
\begin{equation}\label{logar2}
\underset{u\to+\infty}{\lim\sup}\,\frac{Y_{\alpha,\,\beta}(u)}{u^{\alpha+\beta}(\log\log u)^{1-\alpha}}\leq c_{\alpha,\,\beta}\quad\text{a.s.}
\end{equation}
Set $f(u):=u^{\alpha+\beta}(\log\log u)^{1-\alpha}$ for $u\geq e$
and $f(u):=+\infty$ for $u<e$.

\noindent {\sc Case $\beta\geq 0$}. Fix any $c>c_{\alpha,\,\beta}$ and then pick $r>1$ such that $c>r^{\alpha+\beta}c_{\alpha,\,\beta}$. The following is a basic observation for the subsequent proof:
\begin{eqnarray}
-\log\mmp\{Y_{\alpha,\,\beta}(r^n)>cf(r^{n-1})\}&=&-\log\mmp\{Y_{\alpha,\,\beta}(1)>cr^{-(\alpha+\beta)n}f(r^{n-1})\}\notag
\\&\sim&
\bigg(\frac{c}{r^{\alpha+\beta}c_{\alpha,\,\beta}}\bigg)^{(1-\alpha)^{-1}}\log
n,\quad n\to\infty\label{logar4}
\end{eqnarray}
where the equality is a consequence of self-similarity of
$Y_{\alpha,\,\beta}$, and the asymptotic relation follows from \eqref{riv}. Since the factor in front of $\log n$ is greater
than $1$, we infer $\sum_{n\geq 1}\mmp\{Y_{\alpha,\,\beta}(r^n)>cf(r^{n-1})\}<\infty$. The Borel-Cantelli lemma ensures that %, with probability one,
$Y_{\alpha,\,\beta}(r^n)\leq cf(r^{n-1})$ for all $n$ large enough a.s. Since $Y_{\alpha,\,\beta}$ is nondecreasing a.s.\ and $f$ is nonnegative
and increasing on $[e,\infty)$ we have for all large enough $n$
$$Y_{\alpha,\,\beta}(u)\leq Y_{\alpha,\,\beta}(r^n)\leq cf(r^{n-1})\leq cf(u)\quad\text{a.s.}$$ whenever $u\in [r^{n-1},r^n]$.
Hence ${\lim\sup}_{u\to+\infty} Y_{\alpha,\,\beta}(u)/f(u)\leq c$ a.s.\ which proves \eqref{logar2}.

\noindent {\sc Case $\beta\in (-\alpha, 0)$}. In this case
$Y_{\alpha,\,\beta}$ is not monotone which makes the proof more
involved.

Fix any $\varepsilon>0$ and then pick $r>1$ such that
$c_{\alpha,\,\beta}+\varepsilon> r^{\alpha+\beta}
c_{\alpha,\,\beta}$. Suppose we can prove that
\begin{equation}\label{logar3}
I:=\sum_{n\geq 1}\mmp\left\{\sup_{u\in[r^{n-1},\,r^n]}Y_{\alpha,\,\beta}(u)> (c_{\alpha,\,\beta}+2\varepsilon)f(r^{n-1})\right\}<\infty.
\end{equation}
Then, using the Borel-Cantelli lemma we infer % that, with
%probability one,
$$\underset{v\in
[r^{n-1},\,r^n]}{\sup}\,Y_{\alpha,\,\beta}(v)\leq
(c_{\alpha,\,\beta}+2\varepsilon)f(r^{n-1})$$ for all $n$ large
enough a.s. Since $f$ is nonnegative and increasing on
$[e,\infty)$, we have for all large enough $n$
$$Y_{\alpha,\,\beta}(u)\leq \underset{v\in
[r^{n-1},r^n]}{\sup}\,Y_{\alpha,\,\beta}(v)\leq
(c_{\alpha,\,\beta}+2\varepsilon)f(r^{n-1})\leq
(c_{\alpha,\,\beta}+2\varepsilon)f(u)\quad\text{a.s.}$$ whenever
$u\in [r^{n-1},r^n]$. Hence, ${\lim\sup}_{u\to+\infty}
Y_{\alpha,\,\beta}(u)/f(u)\leq c_{\alpha,\,\beta}+2\varepsilon$
a.s.\ which entails \eqref{logar2}.

Let $\lambda>0$ and set $n_r:=[\log^{-1}r]+1$. Passing to the
proof of \eqref{logar3} we have\footnote{For notational
simplicity, we shall write $n^\lambda$ and $n^{-\lambda}$ instead
of $[n^\lambda]$ and $[n^\lambda]^{-1}$ respectively.}
\begin{eqnarray*}
I&=& \sum_{n\geq n_r}\mmp\big\{\sup_{u\in[1/(2r),\,1/2]}Y_{\alpha,\,\beta}(u)>(c_{\alpha,\beta}+2\varepsilon)(2r)^{-(\alpha+\beta)}(\log((n-1)\log r))^{1-\alpha}\big\}\\&\leq& \sum_{n\geq n_r}\sum_{k=1}^{n^\lambda-1}\mmp\big\{\sup_{kn^{-\lambda}/2\leq u\leq (k+1)n^{-\lambda}/2}Y_{\alpha,\,\beta}(u)>(c_{\alpha,\,\beta}+2\varepsilon)(2r)^{-(\alpha+\beta)}(\log((n-1)\log r))^{1-\alpha}\big\}\\
&\leq&\sum_{n\geq n_r}\sum_{k=1}^{n^\lambda-1}\mmp\big\{\sup_{kn^{-\lambda}/2\leq u\leq (k+1)n^{-\lambda}/2}\big|Y_{\alpha,\,\beta}(u)-Y_{\alpha,\,\beta}(kn^{-\lambda}/2)\big|\\
&&\hspace{8cm}> \varepsilon(2r)^{-(\alpha+\beta)}(\log((n-1)\log r))^{1-\alpha}\big\}\\&+& \sum_{n\geq n_r}\sum_{k=1}^{n^\lambda-1}\mmp\left\{Y_{\alpha,\,\beta}(kn^{-\lambda}/2    )>(c_{\alpha,\,\beta}+\varepsilon)(2r)^{-(\alpha+\beta)}(\log((n-1)\log r))^{1-\alpha}\right\}=:I_1+I_2.
\end{eqnarray*}
Using \eqref{aux1}, %in combination with the inequality $x^{(\alpha+\beta)/2}|\log x|^{1-\alpha}\leq c$, which holds for all $x\in (0,1/2]$ and some $c>0$
we infer
$$
\underset{kn^{-\lambda}/2\leq u\leq (k+1)n^{-\lambda}/2}{\sup}\,|Y_{\alpha,\,\beta}(u)-Y_{\alpha,\,\beta}(kn^{-\lambda}/2)|\leq C_1 M(n^{-\lambda}/2)^{(\alpha+\beta)/2}
$$
for $1\leq k\leq n^\lambda-1$, where
$C_1:=2|\beta|\kappa_{\alpha,\,\beta}\sup_{x\in (0,1/2]}
\left(x^{(\alpha+\beta)/2}|\log x|^{1-\alpha}\right)$. Hence,
\begin{eqnarray*}
I_1\leq  \sum_{n\geq n_r} n^\lambda \mmp\left\{M> (\varepsilon/C_1) (2r)^{-(\alpha+\beta)}(2n^\lambda)^{(\alpha+\beta)/2}(\log((n-1)\log r))^{1-\alpha}\right\}<\infty
\end{eqnarray*}
for all $\lambda>0$, where the finiteness is justified by \eqref{sam3} and Markov's inequality.
Further,
\begin{eqnarray*}
I_2&=&\sum_{n\geq n_r}\sum_{k=1}^{n^\lambda-1}\mmp\left\{(kn^{-\lambda}/2)^{\alpha+\beta}Y_{\alpha,\,\beta}(1) \geq (c_{\alpha,\beta}+\varepsilon)(2r)^{-(\alpha+\beta)}(\log((n-1)\log r))^{1-\alpha}\right\}\\
&&\leq\sum_{n\geq n_r} n^\lambda\mmp\left\{Y_{\alpha,\,\beta}(1)
\geq
(c_{\alpha,\beta}+\varepsilon)r^{-(\alpha+\beta)}(\log((n-1)\log
r))^{1-\alpha}\right\}<\infty
\end{eqnarray*}
for all
$$\lambda<\bigg(\frac{c_{\alpha,\beta}+\varepsilon}{r^{\alpha+\beta}c_{\alpha,\,\beta}}\bigg)^{(1-\alpha)^{-1}}-1$$
which is positive by the choice of $r$. Here, the equality follows
by self-similarity of $Y_{\alpha,\,\beta}$ (see Proposition
\ref{simil}) and the finiteness is a consequence of \eqref{logar4}
(with $c_{\alpha,\,\beta}+\varepsilon$ replacing $c$). Thus,
\eqref{logar3} holds, and the proof of \eqref{logar2} is complete.

Now we pass to the proof of the limit relation
\begin{equation}\label{logar5}
\underset{u\to+\infty}{\lim\sup}\,\frac{Y_{\alpha,\,\beta}(u)}{u^{\alpha+\beta}(\log\log u)^{1-\alpha}}\geq c_{\alpha,\,\beta}\quad\text{a.s.}
\end{equation}
To this end, we define $\widetilde{D}_\alpha:=(\widetilde{D}_\alpha(y))_{y\geq 0}$ by $$\widetilde{D}_\alpha(y):=D_\alpha(W_\alpha(1)+y)-D_\alpha(W_\alpha(1)),\quad y\geq 0.$$ By the strong Markov property of $D_\alpha$, the process $\widetilde{D}_\alpha$ is a copy of $D_\alpha$ which is further independent of $(D_\alpha(y))_{0\leq y\leq W_\alpha(1)}$. This particularly implies that $\widetilde{Y}_{\alpha,\,\beta}:=(\widetilde{Y}_{\alpha,\,\beta}(u))_{u\geq 0}$ defined by $$\widetilde{Y}_{\alpha,\,\beta}(u):=\int_0^\infty(u-\widetilde{D}_\alpha(y))^\beta\1_{\{\widetilde{D}_\alpha(y)\leq u\}}{\rm d}y,\quad u\geq 0$$ is a copy of $Y_{\alpha,\,\beta}$ which is independent of $\big(D_\alpha(W_\alpha(1)), \int_0^\infty(v-D_\alpha(y))^\beta\1_{\{D_\alpha(y)\leq 1\}}{\rm d}y\big)$ for each $v\geq 1$.
We shall use the following decomposition
\begin{eqnarray}
Y_{\alpha,\,\beta}(u)&=&\int_0^\infty(u-D_\alpha(y))^\beta\1_{\{D_\alpha(y)\leq u\}}{\rm d}y=
\int_0^\infty(u-D_\alpha(y))^\beta\1_{\{D_\alpha(y)\leq 1\}}{\rm d}y\notag\\&+&
\widetilde{Y}_{\alpha,\,\beta}(u-D_\alpha(W_\alpha(1))\1_{\{D_\alpha(W_\alpha(1))\leq u\}}\label{decomposition_y}
\end{eqnarray}
which holds for $u>1$ and can be justified as follows:
\begin{eqnarray*}
&&\int_0^\infty(u-D_\alpha(y))^\beta\1_{\{1<D_\alpha(y)\leq u\}}{\rm d}y\\&=& \int_0^\infty(u-D_\alpha(y+W_\alpha(1)))^\beta\1_{\{D_\alpha(y+W_{\alpha}(1))\leq u\}}{\rm d}y\notag\\
&=&\int_0^\infty(u-D_\alpha(W_\alpha(1))-\widetilde{D}_\alpha(y))^\beta\1_{\{\widetilde{D}_\alpha(y)\leq u-D_\alpha(W_\alpha(1))\}}{\rm d}y\1_{\{D_\alpha(W_\alpha(1))\leq u\}}\\
&=&\widetilde{Y}_{\alpha,\,\beta}(u-D_\alpha(W_\alpha(1))\1_{\{D_\alpha(W_\alpha(1))\leq u\}}. %\label{decomposition_y},
\end{eqnarray*}
%where the process $\widetilde{Y}_{\alpha,\beta}(\cdot)$ has the same distribution as $Y_{\alpha,\,\beta}(\cdot)$ and is independent of $(D_{\alpha}(y))_{y\in[0,W_{\alpha}(1)]}$ and hence independent of $Y_{\alpha,\,\beta}(1)$. This follows from the strong Markov property of $D_{\alpha}(\cdot)$.

Our proof of \eqref{logar5} will be based on the following extension of the Borel-Cantelli lemma due to Erd\"{o}s and R\'{e}nyi (Lemma C in \cite{Erdos+Renyi:1959}).
\begin{lemma}\label{bc_reverse_lemma}
Let $(A_k)_{k\in\mn}$ be a sequence of random events such that $\sum_{k\geq 1}\mmp\{A_k\}=\infty$. If
$$
\liminf_{n\to\infty}\frac{\sum_{i=1}^n\sum_{j=1}^{n}\mmp\{A_i\cap A_j\}}{\left(\sum_{k=1}^n\mmp\{A_k\}\right)^2}\leq 1,
$$
then $\mmp\{\underset{k\to\infty}{\lim\sup}\, A_k\}=1$.
\end{lemma}

Fix any $c\in (0,c_{\alpha,\,\beta})$ and some $r>1$ to be specified later. %, any $\varepsilon\in (0,c_{\alpha,\,\beta})$ and set $c:=c_{\alpha,\,\beta}-\varepsilon$.
Putting $A_k:=\{Y_{\alpha,\,\beta}(r^k)\geq c f(r^k)\}$ for $k\in\mn$ and using \eqref{riv}, we obtain
$$
-\log\mmp\{A_n\}\quad \sim\quad (c/c_{\alpha,\,\beta})^{(1-\alpha)^{-1}}\log n=:c_0\log n,\quad n\to\infty
$$
which entails $\sum_{k\geq 1}\mmp\{A_k\}=\infty$ because $c_0<1$. Also, for any $\delta>0$ there exists $n_0\in\mn$ such that
\begin{equation}\label{prob_a_k_bounds}
n^{-c_0-\delta}\leq \mmp\{A_n\}\leq n^{-c_0+\delta}
\end{equation}
for all $n\geq n_0$.
%In particular, since $c_0<1$ we have $\sum_{k=1}^{\infty}\mmp\{A_k\}=\infty$.
Now we have to find an appropriate upper bound for
\begin{eqnarray*}
\mmp\{A_i\cap A_{i+n}\}&=&\mmp\bigg\{Y_{\alpha,\,\beta}(1)\geq c (\log(i\log r))^{1-\alpha}, Y_{\alpha,\,\beta}(r^n)\geq c r^{n(\alpha+\beta)}(\log((n+i)\log r))^{1-\alpha}\bigg\}\\
&&\hspace{-2cm}=\mmp\bigg\{Y_{\alpha,\,\beta}(1) \geq c (\log(i\log r))^{1-\alpha},\int_0^\infty (r^n-D_\alpha(y))^\beta\1_{\{D_\alpha(y)\leq 1\}}{\rm d}y\\
&&\hspace{-2cm}+\widetilde{Y}_{\alpha,\,\beta}(r^n-D_{\alpha}(W_{\alpha}(1))\1_{\{D_{\alpha}(W_{\alpha}(1))\leq r^n\}}\geq c r^{n(\alpha+\beta)}(\log((n+i)\log r))^{1-\alpha}\bigg\}\\
&&\hspace{-2cm}=\mmp\bigg\{Y_{\alpha,\,\beta}(1)\geq c (\log(i\log r))^{1-\alpha},\int_0^\infty (r^n-D_\alpha(y))^\beta\1_{\{D_\alpha(y)\leq 1\}}{\rm d}y\\
&&\hspace{-2cm}+(r^n-D_\alpha(W_\alpha(1)))_{+}^{\alpha+\beta}\widetilde{Y}_{\alpha,\,\beta}(1)\geq c r^{n(\alpha+\beta)}(\log((n+i)\log r))^{1-\alpha}\bigg\}\\
&&\hspace{-2cm}\leq \mmp\bigg\{Y_{\alpha,\,\beta}(1)\geq c(\log(i\log r))^{1-\alpha},r^{-\alpha n}\int_0^\infty (1-r^{-n}D_\alpha(y))^\beta\1_{\{D_\alpha(y)\leq 1 \}}{\rm d}y\\
&&\hspace{-2cm}+\widetilde{Y}_{\alpha,\,\beta}(1)\geq c (\log((n+i)\log r))^{1-\alpha}\bigg\}\\
&&\hspace{-2cm}\leq \mmp\bigg\{Y_{\alpha,\,\beta}(1)\geq c (\log(i\log r))^{1-\alpha}, \Delta_n+\widetilde{Y}_{\alpha,\,\beta}(1)\geq c(\log((n+i)\log r))^{1-\alpha}\bigg\}
\end{eqnarray*}
for $i\geq [\log^{-1}r]$ and $n\in\mn$, where $\Delta_n:=\gamma_r r^{-\alpha n}W_\alpha(1)$ and $\gamma_r:=(1-r^{-1})^\beta\vee 1$. For the first equality we have used self-similarity of $Y_{\alpha,\,\beta}$ (see Proposition \ref{simil}); the second equality is equivalent to \eqref{decomposition_y}; the third equality is a consequence of self-similarity of $\widetilde{Y}_{\alpha,\,\beta}$ together with independence of $\widetilde{Y}_{\alpha,\,\beta}$ and all the other random variables which appear in that equality; the last inequality follows from
$$\int_0^\infty (1-r^{-n}D_\alpha(y))^\beta\1_{\{D_\alpha(y)\leq 1\}}{\rm d}y \leq \gamma_r W_\alpha(1).$$
Further,
\begin{eqnarray*}
&& \mmp\{A_i\cap A_{i+n}\}-\mmp\{A_i\}\mmp\{A_{i+n}\}\\&\leq& \mmp\bigg\{c(\log((n+i)\log r))^{1-\alpha}-\Delta_n\leq \widetilde{Y}_{\alpha,\,\beta}(1)< c(\log((n+i)\log r))^{1-\alpha},\\&& \Delta_n\leq c(\log((n+i)\log r))^{1-\alpha}\bigg\}+\mmp\{\Delta_n> c(\log((n+i)\log r))^{1-\alpha}\}\\&:=&J_1(n,i)+J_2(n,i)=:J(n,i)
\end{eqnarray*}
for $i\geq [\log^{-1}r]$ and $n\in\mn$.

Suppose we can prove that
\begin{equation}\label{lil_lower_bound1}
\phi_i:=\sum_{n\geq 1} J(n,i)=O(i^{-c_0+\delta}),\quad i\to\infty
\end{equation}
for $\delta$ in \eqref{prob_a_k_bounds}, which further satisfies $c_0+3\delta<1$. % Indeed, if \eqref{lil_lower_bound1} holds, then
Then
\begin{eqnarray*}
&&\liminf_{n\to\infty}\frac{\sum_{i=1}^n\sum_{j=1}^n\mmp\{A_i\cap A_j\}}{\left(\sum_{k=1}^n\mmp\{A_k\}\right)^2}=\liminf_{n\to\infty}\frac{2\sum_{i=1}^n\sum_{j=1}^{n-i}\mmp\{A_i\cap A_{i+j}\}}{\left(\sum_{k=1}^n\mmp\{A_k\}\right)^2}\\
&&\hspace{1cm}\leq\liminf_{n\to\infty}\frac{2\sum_{i=1}^n\sum_{j=1}^{n-i}\mmp\{A_i\}\mmp\{A_{i+j}\}+2\sum_{i=1}^n\phi_i}{\left(\sum_{k=1}^n\mmp\{A_k\}\right)^2}\leq 1+2\liminf_{n\to\infty}\frac{\sum_{i=1}^
n\phi_i}{\left(\sum_{k=1}^n\mmp\{A_k\}\right)^2}=1,
\end{eqnarray*}
thereby proving that $\mmp\{\underset{k\to\infty}{\lim\sup}\, A_k\}=1$ by Lemma \ref{bc_reverse_lemma}. Thus, $$\underset{u\to+\infty}{\lim\sup}\,\frac{Y_{\alpha,\,\beta}(u)}{u^{\alpha+\beta}(\log\log u)^{1-\alpha}}\geq \underset{k\to\infty}{\lim\sup}\,\frac{Y_{\alpha,\,\beta}(r^k)}{r^{k(\alpha+\beta)}(\log\log r^k)^{1-\alpha}}\geq c$$ which shows that \eqref{lil_lower_bound1} entails \eqref{logar5}.

%In view of \eqref{prob_a_k_bounds} the latter $\liminf$ equals zero as soon as $\delta_1$ is small enough, showing that $\mmp\{\limsup A_k\}=1$ by Lemma \ref{bc_reverse_lemma}.

\noindent {\sc Proof of \eqref{lil_lower_bound1}}. Pick both $\delta_1$ in \eqref{riv2} and some $\varepsilon>0$ so small that
\begin{equation}\label{aux3}
c_0(1-\delta_1)(1-\varepsilon)^{(1-\alpha)^{-1}}=(1-\delta_1)((c/c_{\alpha,\,\beta})(1-\varepsilon))^{(1-\alpha)^{-1}}\geq c_0-\delta.
\end{equation}
Using now \eqref{riv2} and recalling that the density $f_{\alpha,\,\beta}$ of $\widetilde{Y}_{\alpha,\,\beta}$ is nonincreasing we infer
\begin{eqnarray*}
J_1(n,i)&\leq&\me \left(\Delta_n f_{\alpha,\,\beta}\left(c (\log((n+i)\log r))^{1-\alpha}-\Delta_n\right)\right)\1_{\{\Delta_n\leq c(\log((n+i)\log r))^{1-\alpha}\}}\\
&\leq& c_1 \me \left(\Delta_n \exp\left(-(1-\delta_1)c_{\alpha,\,\beta}^{-(1-\alpha)^{-1}}\left(\left(c(\log((n+i)\log r))^{1-\alpha}-\Delta_n\right)^{(1-\alpha)^{-1}}\right)\right)\right)\\&\times& (\1_{\{\Delta_n\leq \varepsilon c(\log((n+i)\log r))^{1-\alpha}\}}+\1_{\{\Delta_n>\varepsilon c(\log((n+i)\log r))^{1-\alpha}\}})\\&=&J_{11}(n,i)+J_{12}(n,i).
\end{eqnarray*}
An application of \eqref{aux3} yields
\begin{eqnarray*}
J_{11}(n,i)\leq c_1 \me \big(\Delta_n \exp\big(-(c_0-\delta)\log((n+i)\log r)\big)\leq \frac{c_1}{(\log r)^{c_0-\delta}}\frac{\me\Delta_n}{ (n+i)^{c_0-\delta}}.
\end{eqnarray*}
In view of \eqref{riv} with $\beta=0$, for any $\delta_2\in (0,1)$ there exists $c_2>0$ such that
\begin{equation}\label{aux4}
\mmp\{W_\alpha(1)>x\}\leq c_2\exp\big(-(1-\delta_2)(x/c_{\alpha,\,\beta})^{(1-\alpha)^{-1}}\big)
\end{equation}
for all $x\geq 0$. Let $r>1$ satisfy $\varepsilon r^\alpha>1$ with $\varepsilon$ as in \eqref{aux3}. With this choice of $r$, we can pick $\delta_2>0$ so small and $q>1$ so close to $1$ that $$c_0(1/q)(1-\delta_2)(\varepsilon\gamma_rr^\alpha)^{(1-\alpha)^{-1}}=(1/q)(1-\delta_2)(\varepsilon\gamma_r(c/c_{\alpha,\,\beta})r^\alpha )^{(1-\alpha)^{-1}}\geq c_0-\delta.$$ Then, using H\"{o}lder's inequality with $q$ as above and $p>1$ satisfying $1/p+1/q=1$ gives
\begin{eqnarray*}
J_{12}(n,i)&\leq& c_1 \me \Delta_n\1_{\{\Delta_n > \varepsilon c (\log((n+i)\log r))^{1-\alpha}\}}\\&\leq& c_1(\me \Delta_n^p)^{1/p}
\left(\mmp\{W_\alpha(1) > \varepsilon\gamma_r c r^{\alpha  n}(\log((n+i)\log r))^{1-\alpha}\}\right)^{1/q}\\&\leq&
c_1c_2^{1/q}(\me \Delta_n^p)^{1/p}\exp\big(-(1/q)(1-\delta_2)(\varepsilon\gamma_r(c/c_{\alpha,\,\beta})r^{\alpha n})^{(1-\alpha)^{-1}}\log((n+i)\log r)\big)\\&\leq& \frac{c_1c_2^{1/q}}{ (\log r)^{c_0-\delta}}\frac{(\me \Delta_n^p)^{1/p}}{(n+i)^{c_0-\delta}}.
\end{eqnarray*}
Put $c_3:=c_1(c_2^{1/q}+1)(\log r)^{-c_0+\delta}$. Since $c_4:=\sum_{n\geq 1}\left(\me \Delta_n^p \right)^{1/p}<\infty$, we infer $$\sum_{n\geq 1}J_1(n,i)\leq c_3\sum_{n\geq 1}\frac{\left(\me \Delta_n^p \right)^{1/p}}{ (n+i)^{c_0-\delta}}\leq \frac{c_3c_4}{i^{c_0-\delta}}$$ for each $i\in\mn$.   %, where $n_1$ is the minimal integer that satisfies $(1/2)(1-\delta_3)c_{\alpha,\,\beta}^{-(1-\alpha)^{-1}}(\varepsilon\gamma_rcr^{\alpha n_1})^{(1-\alpha)^{-1}}\geq c_0-\delta$.
It remains to treat $J_2(n,i)$. Increasing $r$ if needed, we can assume that
$$(1-\delta_2)(\gamma_r(c/c_{\alpha,\,\beta})r^\alpha )^{(1-\alpha)^{-1}}\geq 2.$$
Then, in view of \eqref{aux4}, $$J_2(n,i)\leq
c_2\exp\big(-(1-\delta_2)(\gamma_r(c/c_{\alpha,\,\beta})r^{\alpha
n})^{(1-\alpha)^{-1}}\log((n+i)\log r)\big)\leq \frac{c_2}{(\log
r)^2}\frac{1}{(n+i)^2},$$ whence
$$\sum_{n\geq 1}J_2(n,i)\leq \frac{c_2}{(\log r)^2}\sum_{n\geq i+1}\frac{1}{n^2}=O(i^{-1})=O(i^{-c_0+\delta}).$$
Thus, relation \eqref{lil_lower_bound1} has been checked, and the proof of the law of iterated logarithm for large times is complete.

A perusal of the proof above reveals that the proof for small
times can be done along similar lines. When defining sequences
$(r^n)$ just take $r\in (0,1)$ rather than $r>1$. Self-similarity
of $Y_{\alpha,\,\beta}$ does the rest. %: \textcolor{green}{one
%should take $r\in (0,1)$ rather than $r>1$, and self-similarity of
%$Y_{\alpha,\,\beta}$ does the rest}.
We omit further details.

\section{Appendix}

Lemma \ref{rivero} is a consequence of Proposition 2 in \cite{Rivero:2003} and Corollary 2.2 in \cite{Pardo+Rivero+Schaik:2013}.
\begin{lemma}\label{rivero}
For $R:=(R(t))_{t\geq 0}$ a subordinator with positive killing
rate, the random variable $\int_0^\infty \exp(-R(t)){\rm d}t$ has
bounded and nonincreasing density $f$. If the Laplace exponent
$\Psi$ of $R$ is regularly varying at $\infty$ of index $\gamma\in
(0,1)$, then
\begin{equation*}\label{riv11}
-\log\mmp\bigg\{\int_0^\infty \exp(-R(t)){\rm d}t>x\bigg\}\quad \sim\quad -\log f(x)\quad\sim\quad (1-\gamma)\Phi(x),\quad x\to\infty,
\end{equation*}
where $\Phi(t)$ is generalized inverse of $t/\Psi(t)$. %Furthermore, the random variable $\int_0^\infty \exp(-R(t)){\rm d}t$ has a bounded and nonincreasing density $f$ which satisfies
%\begin{equation}\label{riv12}
%-\log\mmp\bigg\{\int_0^\infty \exp(-R(t)){\rm d}t>x\bigg\}\quad \sim\quad -\log f(x),\quad x\to\infty.
%\end{equation}
\end{lemma}

The following result is an important ingredient for the proofs of
Theorems \ref{xxx} and \ref{fin-dim}.
\begin{lemma}\label{continuity_of_convolution_in_D}
Assume that $f_n$ are right-continuous and nondecreasing for each
$n\in\mn_0$ and that $\lim_{n\to\infty} f_n= f_0$ locally
uniformly on $[0,\infty)$. Then, for any $\varepsilon\in(0,1)$ and
any $\beta\in\mr$,
$$
\lim_{n\to\infty} \int_{[0,\,\varepsilon u]}(u-y)^\beta {\rm
d}f_n(y)= \int_{[0,\,\varepsilon u]}(u-y)^{\beta}{\rm d}f_0(y)
$$
locally uniformly on $(0,\infty)$.
\end{lemma}
\begin{proof}
Fix positive $a<b$. Integrating by parts, we obtain
\begin{eqnarray*}
\int_{[0,\,\varepsilon u]}(u-y)^\beta{\rm
d}f_n(y)&=&(1-\varepsilon)^\beta u^\beta f_n(\varepsilon
u)-u^\beta f_n(0)+\beta\int_0^{\varepsilon
u}(u-y)^{\beta-1}f_n(y){\rm d}y
\end{eqnarray*}
for $n\in\mn_0$. The claim follows from the relations
\begin{eqnarray*}
\sup_{u\in[a,\,b]}\left|u^\beta f_n(\varepsilon u)-u^\beta f_0(\varepsilon u)\right|&\leq& (a^{\beta}\vee b^{\beta})\sup_{u\in[0,\,b]}
\left|f_n(u)-f_0(u)\right|\to 0;\\
\sup_{u\in[a,\,b]}\left|u^{\beta}f_n(0)-u^{\beta}f_0(0)\right|&\leq&
(a^{\beta}\vee b^{\beta})\left|f_n(0)-f_0(0)\right|\to 0
\end{eqnarray*}
and
\begin{eqnarray*}
&&\sup_{u\in[a,\,b]}\bigg|\int_0^{\varepsilon
u}(u-y)^{\beta-1}f_n(y){\rm d}y-\int_0^{\varepsilon
u}(u-y)^{\beta-1}f_0(y){\rm d}y\bigg|\\&\leq&
\sup_{u\in[a,\,b]}\int_0^{\varepsilon
u}(u-y)^{\beta-1}\left|f_n(y)-f_0(y)\right|{\rm
d}y\\&\leq&\sup_{u\in[0,\,b]}\left|f_n(u)-f_0(u)\right|\sup_{u\in[a,\,b]}\int_0^{\varepsilon
u}(u-y)^{\beta-1}{\rm
d}y\\&=&\sup_{u\in[0,\,b]}\left|f_n(u)-f_0(u)\right|(a^{\beta}\vee
b^{\beta})|\beta|^{-1}|1-(1-\varepsilon)^\beta|\to 0
\end{eqnarray*}
as $n\to\infty$.
\end{proof}

Recall that $(\nu(t))_{t\geq 0}$ is the first-passage time process defined by $\nu(t)=\inf\{k\in\mn: S_k>t\}$ for $t\geq 0$, where $(S_k)_{k\in\mn_0}$ is a zero-delayed standard random walk with jumps distributed as a positive random variable $\xi$. Lemma \ref{iks13} is Lemma A.1 in \cite{Iksanov:2013}.
\begin{lemma}\label{iks13}
For any finite $d>c\ge 0$, any $T>0$ and any $r>0$
$$t^{-r}\underset{u\in [0,\,T]}{\sup}\,\big(\nu(ut-c)-\nu(ut-d)\big)\quad\tp\quad 0,\quad t\to\infty.$$
\end{lemma}

The two results given next are needed for the proof of Proposition \ref{renewal_modulus_of_continuity}

\begin{lemma}\label{exp_mom_nu_lemma}
Assume that $\mmp\{\xi>t\}\sim t^{-\alpha}\ell(t)$ for some
$\alpha\in (0,1)$ and some $\ell$ slowly varying at infinity. Then
$\sup_{t\geq 0}\me e^{\lambda\mmp\{\xi>t\}\nu(t)}<\infty$ for
every $\lambda>0$.
%$$
%\lim_{t\to\infty}\me e^{\lambda a(t)\nu(t)}=\me e^{\lambda W_{\alpha}(1)}=E_{\alpha,1}\left(\frac{\lambda}{\Gamma(1-\alpha)}\right)<\infty,
%$$
%where $E_{\alpha,1}(z):=\sum_{n=0}^{\infty}\frac{z^k}{\Gamma(n\alpha+1)}$, $z\in\mathbb{C}$, is the Mittag-Leffler function.
\end{lemma}
\begin{proof}
As before, we shall use the notation $a(t)=\mmp\{\xi>t\}$. Fix any
$\lambda>0$. Since $\nu(t)$ has finite exponential moments of all
orders for all $t\geq 0$, it suffices to show that
$$
\underset{t\to\infty}{\lim\sup}\,\me e^{\lambda
a(t)\nu(t)}<\infty.
$$
%Since $e^{\lambda a(t)\nu(t)}\dod e^{\lambda W_{\alpha}(1)}$, it is enough to prove that the family $(e^{\lambda a(t)\nu(t)})_{t\geq 0}$ is uniformly integrable. By a Vall\'{e}e-Poussin theorem, the latter is guaranteed by $\sup_{t\geq 0}\me e^{2\lambda a(t)\nu(t)}<\infty$. Summarizing, it suffices to show that the family $(\me e^{\lambda a(t)\nu(t)})_{t\geq 0}$ is bounded for every $\lambda>0$.
We have
\begin{eqnarray*}
\frac{\me e^{\lambda a(t)\nu(t)}-1}{e^{\lambda a(t)}-1}e^{\lambda a(t)}&=&\sum_{k\geq 1}e^{\lambda a(t) k}\mmp\{\nu(t)\geq k\}=\sum_{k\geq 1}e^{\lambda a(t) k}\mmp\{S_{k-1}\leq t\}\\
&=&\sum_{k\geq 1}e^{\lambda a(t) k}\mmp\{e^{-sS_{k-1}}\geq e^{-st}\}\leq e^{st}\sum_{k\geq 1}e^{\lambda a(t) k}(\phi(s))^{k-1}\\
&=&\frac{e^{st}e^{\lambda a(t)}}{1-e^{\lambda a(t)}\phi(s)}
\end{eqnarray*}
for any $s>0$ such that $e^{\lambda a(t)}\phi(s)<1$. Pick an
arbitrary $c>(\lambda/\Gamma(1-\alpha))^{1/\alpha}$ and note that
\begin{equation}\label{exp_mom_nu_proof1}
\frac{1-e^{-\lambda a(t)}}{1-\phi(c/t)}\sim \frac{\lambda a(t)}{\Gamma(1-\alpha)a(t/c)}\to \frac{\lambda c^{-\alpha}}{\Gamma(1-\alpha)}<1
\end{equation}
as $t\to\infty$, where the asymptotics $1-\phi(z)\sim \Gamma(1-\alpha)a(1/z)$ as $z\to 0+$ follows from Karamata's Tauberian theorem (Theorem 1.7.1 in
\cite{BGT}). From \eqref{exp_mom_nu_proof1} we infer $e^{\lambda a(t)}\phi(c/t)<1$ for all $t>0$ large enough. Therefore,
$$
\me e^{\lambda a(t)\nu(t)}-1\leq e^c\frac{e^{\lambda a(t)}-1}{1-e^{\lambda a(t)}\phi(c/t)}.$$ Since, by \eqref{exp_mom_nu_proof1}, the right-hand side converges to $\frac{e^c\lambda}{\Gamma(1-\alpha)c^{\alpha}-\lambda}$ as $t\to\infty$, the proof of Lemma \ref{exp_mom_nu_lemma} is complete.
\end{proof}
The following slightly strengthened version of Potter's bound (Theorem 1.5.6 in \cite{BGT}) takes advantage of additional monotonicity.
\begin{lemma}\label{potter_bound2}
Let $f:[0,\infty)\to(0,\infty)$ be a nonincreasing function which is regularly varying at $\infty$ of index $-\rho<0$. Then, for any chosen $\gamma\in(0,\rho)$ and $x_0>0$, there exist $t_0>0$ and $c>0$ such that
\begin{equation}\label{potter_bound2_ineq}
f(y)/f(x)\geq c(x/y)^{\rho-\gamma}
\end{equation}
for all $x\geq t_0$ and all $y\in[x_0,x]$.
\end{lemma}
\begin{proof}
Fix $\gamma\in (0,\rho)$, $x_0>0$ and $c_1>0$. By Potter's bound, there exists $t_0> x_0$ such that
%$$
%\frac{f(x)}{f(y)}\leq c_1\left(\frac{x}{y}\right)^{-\rho+\delta}
%$$
%and therefore, for all $x\geq y\geq t_0$,
$$
f(y)/f(x)\geq c_1(x/y)^{\rho-\gamma}
$$
for all $x\geq y\geq t_0$.
On the other hand, monotonicity of $f$ entails
$$
f(y)/f(x)\geq f(t_0)/f(x)\geq c_1(x/t_0)^{\rho-\gamma}\geq c_1(x_0/t_0)^{\rho-\gamma}(x/y)^{\rho-\gamma}
$$
for $x\geq t_0$ and $y\in[x_0,t_0)$, and \eqref{potter_bound2_ineq} follows upon setting $c:=c_1(x_0/t_0)^{\rho-\gamma}$.
\end{proof}

\begin{proof}[Proof of Proposition \ref{renewal_modulus_of_continuity}]
Since $a(t)=\mmp\{\xi>t\}$ is regularly varying, we can assume that $T=1$. We start by noting that (see Figure~\ref{square_division})
\begin{eqnarray*}
&&\hspace{-1.5cm}\sup_{(u,v)\in A_t}\frac{a(t)(\nu(ut)-\nu(vt))}{(u-v)^{\alpha-\delta}}\leq \sup_{1/t\leq h\leq 1}\sup_{0\leq u\leq 1}\frac{a(t)(\nu(ut)-\nu((u-h)t))}{h^{\alpha-\delta}}\\
&&\hspace{-1cm}\leq \sup_{j=1,\ldots,\lceil\log_2 t\rceil}\sup_{2^{-j}\leq h\leq 2^{-j+1}}\sup_{k=1,\ldots,2^{j-1}}\sup_{(k-1)2^{-j+1}\leq u\leq k2^{-j+1}}\frac{a(t)(\nu(ut)-\nu((u-h)t))}{h^{\alpha-\delta}}\\
&&\hspace{-1cm}\leq \sup_{j=1,\ldots,\lceil\log_2
t\rceil}\sup_{k=1,\ldots,2^{j-1}}\frac{a(t)(\nu(tk2^{-j+1})-\nu(t((k-2)2^{-j+1})))}{2^{-j(\alpha-\delta)}}
\end{eqnarray*}
having utilized a.s.\ monotonicity of $(\nu(t))_{t\geq 0}$ for the
last inequality. Here, $\lceil\cdot\rceil$ denotes the ceiling
function. 

\begin{figure}
\begin{center}
\includegraphics[scale=0.4]{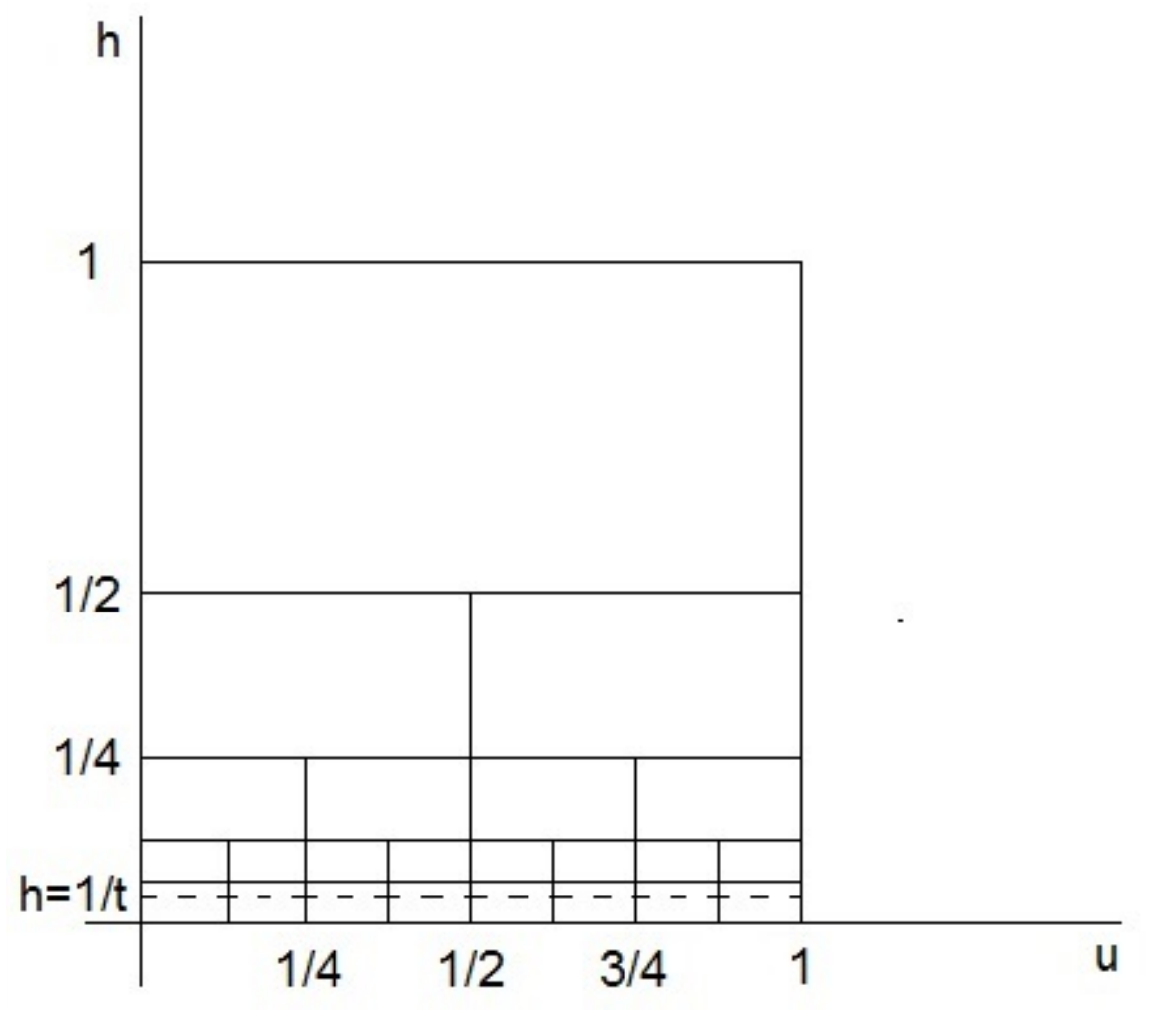}
\end{center}
\caption{Square division in the proof of Proposition \ref{renewal_modulus_of_continuity}}
\label{square_division}
\end{figure}

An application of Boole's inequality yields
\begin{eqnarray*}
&&\mmp\left\{\sup_{(u,v)\in A_t}\frac{a(t)(\nu(ut)-\nu(vt))}{(u-v)^{\alpha-\delta}} > x\right\}\\
&&\leq\sum_{j=1}^{\lceil\log_2 t\rceil}\sum_{k=1}^{2^{j-1}}\mmp\left\{\frac{a(t)(\nu(tk2^{-j+1})-\nu(t((k-2)2^{-j+1})))}{2^{-j(\alpha-\delta)}} > x\right\}.
\end{eqnarray*}
By distributional subadditivity (see formula (5.7) on p.~58 in
\cite{Gut:2009}) of $(\nu(t))_{t\geq 0}$ (for $k\geq 3$) and by
monotonicity of $(\nu(t))_{t\geq 0}$ (for $k=1,2$)
\begin{eqnarray*}
\mmp\left\{\frac{a(t)(\nu(tk2^{-j+1})-\nu(t((k-2)2^{-j+1})))}{2^{-j(\alpha-\delta)}}> x\right\}\leq \mmp\{a(t)\nu(t2^{-j+2})> x2^{-j(\alpha-\delta)}\}
\end{eqnarray*}
whence
\begin{eqnarray*}
&&\hspace{-1cm}\mmp\left\{\sup_{(u,v)\in A_t}\frac{a(t)(\nu(ut)-\nu(vt))}{(u-v)^{\alpha-\delta}}> x\right\}\leq \sum_{j=1}^{\lceil\log_2 t\rceil}2^{j-1}\mmp\{a(t)\nu(t2^{-j+2})> x2^{-j(\alpha-\delta)}\}\\
&&=\sum_{j=1}^{\lceil\log_2 t\rceil}2^{j-1}\mmp\{\exp(a(t2^{-j+2})\nu(t2^{-j+2}))> \exp(x2^{-j(\alpha-\delta)}a(t2^{-j+2})/a(t))\}\\
&&\leq \sum_{j=1}^{\lceil\log_2 t\rceil}2^{j-1}\exp(-x2^{-j(\alpha-\delta)}a(t2^{-j+2})/a(t)) \, \me \exp(a(t2^{-j+2})\nu(t2^{-j+2}))\\
&&\leq C \sum_{j=1}^{\lceil\log_2 t\rceil}2^{j-1}\exp(-x2^{-j(\alpha-\delta)}a(t2^{-j+2})/a(t)),
\end{eqnarray*}
where the penultimate line is a consequence of Markov's inequality, and the boundedness of $\me \exp(a(t2^{-j+2})\nu(t2^{-j+2}))$ follows from Lemma \ref{exp_mom_nu_lemma}. Applying Lemma \ref{potter_bound2} to the function $a$ we obtain
$$
a(t2^{-j+2})/a(t)\geq c2^{(j-2)(\alpha-\delta/2)}
$$
for some $c>0$, all $t>0$ large enough and all $j=2,\ldots,\lceil\log_2 t\rceil$. Hence,
$$
\limsup_{t\to\infty}\mmp\left\{\sup_{(u,v)\in
A_t}\frac{a(t)(\nu(ut)-\nu(vt))}{(u-v)^{\alpha-\delta}}>
x\right\}\leq C\bigg(\exp(-x2^{\delta- \alpha})+ \sum_{j\geq
2}2^{j-1}\exp(-c_1x2^{\delta j/2})\bigg),
$$
where $c_1:=c2^{\delta-2\alpha}>0$. The last series converges uniformly in $x\in [1,\infty)$. Sending $x\to\infty$ finishes the proof.
\end{proof}

\vspace{1cm} \noindent   {\bf Acknowledgments}  \quad
\footnotesize A part of this work was done while A.~Iksanov was
visiting University of M\"{u}nster in January--February 2016. He
gratefully acknowledges hospitality and the financial support by
DFG SFB 878 ``Geometry, Groups and Actions''. The work of
A.~Marynych was supported by the Alexander von Humboldt
Foundation.


\normalsize








\begin{thebibliography}{99}







\bibitem{Alsmeyer+Iksanov+Marynych:2016} {\sc Alsmeyer, G., Iksanov, A. and Marynych, A.} (2016+). Functional limit theorems for the number of occupied boxes in the Bernoulli sieve. Submitted. Preprint available at {\tt http://arxiv.org/abs/1601.04274}


\bibitem{Baeumer+Meerschaert+Nane:2009} {\sc Baeumer, B., Meerschaert, M. and Nane, E.} (2009). Space-time duality for fractional diffusion. {\em J. Appl. Probab.} {\bf 46}, 1100--1115.

\bibitem{Bertoin:1999} {\sc Bertoin, J.} (1999). {\it Subordinators:
Examples and applications}. LNM {\bf 1717}, P.Bernard (Ed.),
1--91. Berlin: Springer-Verlag.

\bibitem{Bertoin+Yor:2005} {\sc Bertoin J. and Yor, M.} (2005). Exponential functionals
of L\'{e}vy processes. {\em Probability Surveys}. \textbf{2}, 191--212.

\bibitem{Billingsley:1999} {\sc Billingsley, P.} (1999). {\it Convergence of probability measures}, 2nd edition. New York: John Wiley and Sons.

\bibitem{BGT}{\sc Bingham N.~H., Goldie C.~M., and Teugels, J.~L.} (1989).
{\it Regular variation}. Cambridge: Cambridge University Press.

\bibitem{Doney+Obrien:1991} {\sc Doney, R.~A. and O'Brien, G.~L.}
(1991). Loud shot noise. {\em Ann. Appl. Probab.} {\bf 1},
88--103.

\bibitem{Erdos+Renyi:1959}{\sc Erd\"{o}s, P. and R\'{e}nyi, A.} (1959). On Cantor's series with convergent $\sum 1/q_n$. {\em Ann. Univ. Sci. Budapest. E\"{o}tv\"{o}s. Sect. Math.} {\bf 2}, 93--109.

\bibitem{Fristedt:1979}{\sc Fristedt, B.} (1979). Uniform local behavior of stable subordinators. {\em Ann. Probab.} {\bf 7}, 1003--1013.

\bibitem{Gut:2009} {\sc Gut, A.} (2009). {\it Stopped random walks. Limit theorems and applications}, 2nd edition. New York: Springer.

\bibitem{Jung+Owada+Samorodnitsky:2016+} {\sc Jung, P., Owada, T. and Samorodnitsky, G.} (2016+). Functional central limit theorem for negatively dependent heavy-tailed stationary infinitely divisible processes generated by conservative flows. Preprint available at {\tt http://arxiv.org/abs/1504.00935}

\bibitem{Hawkes:1971} {\sc Hawkes, J.} (1971). A lower Lipschitz condition for the stable
subordinator. {\em Z. Wahrscheinlichkeitstheorie Verw. Geb.} {\bf 17}, 23--32.

\bibitem{Homble+McCormick:1995} {\sc Homble, P. and McCormick,
W.~P.} (1995). Weak limit results for the extremes of a class of
shot noise processes. {\em J. Appl. Probab.} {\bf 32}, 707--726.

\bibitem{Hsing+Teugels:1989} {\sc Hsing, T. and Teugels, J.~L.} (1989). Extremal properties of shot noise processes.
{\em Adv. Appl. Probab.} {\bf 21}, 513--525.

\bibitem{Iksanov:2013} {\sc Iksanov, A.} (2013). Functional limit theorems for renewal shot noise processes with increasing response functions.
{\em Stoch. Proc. Appl.} {\bf 123}, 1987--2010.

\bibitem{Iksanov:2013b} {\sc Iksanov, A.} (2013). On the number of empty boxes in the Bernoulli sieve I. {\em Stochastics: An International Journal of
Probability and Stochastic Processes}. {\bf 85}, 946--959.

\bibitem{Iksanov+Kabluchko+Marynych:2016} {\sc Iksanov A., Kabluchko, Z. and Marynych, A.} (2016+). Weak convergence of renewal shot noise processes in the case of slowly varying normalization. Submitted. Preprint available at {\tt http://arxiv.org/abs/1507.02526}

\bibitem{Iksanov+Marynych+Meiners:2013} {\sc Iksanov, A., Marynych, A. and Meiners, M.}
(2013). Limit theorems for renewal shot noise processes with
decreasing response functions. Extended preprint version of
\cite{Iksanov+Marynych+Meiners:2014} at {\tt
http://arxiv.org/abs/arXiv:1212.1583v2}.

\bibitem{Iksanov+Marynych+Meiners:2014} {\sc Iksanov, A., Marynych, A. and Meiners, M.} (2014). Limit
theorems for renewal shot noise processes with eventually
decreasing response functions. {\em Stoch. Proc. Appl.} {\bf 124},
2132--2170.

\bibitem{Iksanov+Marynych+Meiners:2016} {\sc Iksanov, A., Marynych, A. and Meiners, M} (2016). Asymptotics of random processes with immigration I: Scaling limits. {\em Bernoulli}, to appear.

\bibitem{Iksanov+Marynych+Meiners:2016b} {\sc Iksanov, A., Marynych, A. and Meiners, M} (2016). Asymptotics of random processes with immigration II: Convergence to stationarity. {\em Bernoulli}, to appear.

\bibitem{Lamperti:1972} {\sc Lamperti, J.} (1972). Semi-stable Markov processes. {\em Z. Wahrscheinlichkeitstheorie Verw. Geb.} {\bf 22}, 205–-225.

\bibitem{Lebedev:2002} {\sc Lebedev, A.~V.} (2002). Extremes of
subexponential shot noise. {\em Math. Notes.} {\bf 71}, 206--210.

\bibitem{Magdziarz+Schilling:2015}{\sc Magdziarz, M. and Schilling, R.~L.} (2015). Asymptotic properties of Brownian motion delayed by inverse subordinators. {\em Proc. Amer. Math. Soc.} {\bf 143}, 4485--4501.

\bibitem{Magdziarz+Weron:2011}{\sc Magdziarz, M. and Weron, A.} (2011). Ergodic properties of anomalous diffusion processes. {\em Ann. Phys.} {\bf 326}, 2431--2443.


\bibitem{McCormick:1997} {\sc McCormick, W.~ P.} (1997). Extremes for shot noise processes with
heavy tailed amplitudes. {\em J. Appl. Probab.} {\bf 34}, 643--656.

\bibitem{McCormick+Seymour:2001} {\sc McCormick, W.~P. and Seymour, L.} (2001). Extreme values for a class of shot-noise processes.
In Selected Proceedings of the Symposium on Inference for
Stochastic Processes, 33--46, Institute of Mathematical Statistics
Lecture Notes - Monograph Series.

\bibitem{Meerschaert+Benson+Scheffler+Baeumer:2002} {\sc Meerschaert, M., Benson, D., Scheffler, H.-P. and Baeumer, B} (2002). Stochastic solution of space--time fractional diffusion equations. {\em Phys. Rev. E.} {\bf 63}, 1103--1106.

\bibitem{Meerschaert+Nane+Vellaisamy:2009} {\sc Meerschaert, M., Nane, E. and Vellaisamy, P.} (2009). Fractional Cauchy problems on bounded domains. {\em Ann. Probab.} {\bf 37}, 979--1007.

\bibitem{Meerschaert+Scheffler:2004} {\sc Meerschaert, M. and Scheffler, H.-P.} (2004). Limit theorems for continuous-time random walks with
infinite mean waiting times. {\em J. Appl. Probab.} {\bf 41}, 623--638.

\bibitem{Meerschaert+Straka:2013} {\sc Meerschaert, M. and Straka, P.} (2013). Inverse stable subordinators. {\em Math. Model. Nat. Phenom.} {\bf 8}, 1--16.

\bibitem{Nane:2009} {\sc Nane, E.} (2009). Laws of the iterated logarithm for a class of iterated processes. {\em Stat. Prob. Letters} {\bf 79}, 1744--1751.

\bibitem{Owada+Samorodnitsky:2014+} {\sc Owada, T. and Samorodnitsky,
G.} (2015). Functional central limit theorem for heavy tailed
stationary infinitely divisible processes generated by
conservative flows. {\em Ann. Probab.} {\bf 43}, 240--285.
%
%\bibitem{Ross+Samko+Love:1994}{\sc Ross, B., Samko, St.~G. and Love, E.~R.} (1994/5). Functions that have no first order derivative might have
%fractional derivatives of all orders less than one. {\em Real
%Analysis Exchange} {\bf 20}, 140--157.

\bibitem{Pardo+Rivero+Schaik:2013}{\sc Pardo, J.~C., Rivero, V. and van Schaik, K.} (2013). On the density of exponential functionals of L\'{e}vy processes. {\em Bernoulli}. {\bf 19}, 1938--1964.

\bibitem{Rivero:2003} {\sc Rivero, V.} (2003). A law of iterated logarithm for increasing self-similar Markov processes. {\em Stochastics and Stochastics Reports}. {\bf 75}, 443–-472.

\bibitem{Samko+Kilbas+Marichev:1993} {\sc Samko, St.~G., Kilbas, A.~A. and Marichev,
O.~I.} (1993). {\it Fractional integrals and derivatives: theory
and applications}. New York: Gordon and Breach.

\bibitem{Scalas+Viles:2014} {\sc Scalas, E. and Viles, N.} (2014). A functional limit theorem for stochastic integrals driven by a time-changed symmetric $\alpha$-stable L\'{e}vy process. {\em Stoch. Proc. Appl.} {\bf 124}, 385--410.

\bibitem{Stanislavsky+Weron+Weron:2008} {\sc Stanislavsky, A., Weron, K. and Weron, A.} (2008). Diffusion and relaxation controlled by
tempered $\alpha$-stable processes. {\em Phys. Rev. E.} {\bf 78},
051106.

\bibitem{Whitt:2002} {\sc Whitt, W.} (2002). {\it Stochastic-process limits: an introduction to stochastic-process limits and their application to queues}. New York: Springer-Verlag.

%\bibitem{Winkel:2005} {\sc Winkel, M.} (2005). Electronic foreign-exchange markets and passage events of independent subordinators. {\em J. Appl. Probab.} {\bf 42}, 138--152.

\end{thebibliography}
\end{document}